\documentclass{amsart}

\usepackage{ae,aecompl}
\usepackage{esint}
\usepackage{graphicx}
\usepackage[all,cmtip]{xy}
\usepackage{amsmath, amscd}
\usepackage{booktabs}
\usepackage{amsthm}
\usepackage{amssymb}
\usepackage{amsfonts}
\usepackage{qsymbols}
\usepackage{latexsym}
\usepackage{mathrsfs}
\usepackage{cite}
\usepackage{color}
\usepackage{url}
\usepackage{enumerate}
\usepackage[english]{babel}

\usepackage{verbatim}
\usepackage[draft=false, colorlinks=true]{hyperref}

\allowdisplaybreaks

\newcommand {\A}{\mathcal{A}}

\newcommand {\bmo}{\mathrm{bmo}}

\newcommand {\C}{\mathbb C}
\newcommand {\Ca}{\mathcal{C}}

\newcommand {\Da}{\mathcal{D}}

\newcommand {\ud}{\mathrm{d}}

\newcommand {\veps}{\varepsilon}

\newcommand {\F}{\mathcal{F}}

\newcommand {\HT}{\mathcal{H}}
\newcommand {\Hp}{\mathcal{H}^{p}_{FIO}(\Rn)}

\newcommand {\Hps}{\mathcal{H}^{s,p}_{FIO}(\Rn)}
\newcommand {\HpsM}{\mathcal{H}^{s,p}_{FIO}(M)}

\newcommand {\ind}{\mathbf{1}}
\newcommand {\inj}{\mathrm{inj}}

\newcommand {\ka}{\kappa}

\newcommand {\la}{\lambda}
\newcommand {\rb}{\rangle}
\newcommand {\lb}{\langle}
\newcommand {\La}{\mathcal{L}}
\newcommand {\loc}{\mathrm{loc}}

\newcommand {\Ma}{\mathcal{M}}
\newcommand {\N}{\mathbb N}

\newcommand {\ph}{\varphi}
\newcommand {\psit}{\tilde{\psi}}

\newcommand {\R}{\mathbb R}

\newcommand {\Rn}{\mathbb{R}^{n}}

\newcommand {\rank}{\mathrm{rank}}

\newcommand {\supp}{\mathrm{supp}}

\newcommand {\Sw}{\mathcal{S}}

\newcommand {\w}{\omega}

\newcommand {\Z}{\mathbb Z}

\newcommand {\vanish}[1]{\relax}

\newcommand{\wh}{\widehat}
\newcommand{\wt}{\widetilde}

\DeclareMathOperator{\Real}{Re}

\newtheorem{theorem}{Theorem}[section]
\newtheorem{lemma}[theorem]{Lemma}
\newtheorem{proposition}[theorem]{Proposition}
\newtheorem{corollary}[theorem]{Corollary}

\theoremstyle{definition}
\newtheorem{definition}[theorem]{Definition}
\newtheorem{remark}[theorem]{Remark}

\numberwithin{equation}{section}
\protected\def\ignorethis#1\endignorethis{}
\let\endignorethis\relax

\title[Spherical maximal functions and Hardy spaces for FIOs]
	{Spherical maximal functions and Hardy spaces for Fourier integral operators}

\author{Abhishek Ghosh}
\address{Department of Mathematics\\
Indian Institute of Technology Madras\\
Chennai\\
Tamil Nadu\\
600036, India}
\email{abhi@iitm.ac.in}

\author{Naijia Liu}
\address{Department of Mathematics,
    Sun Yat-sen University,
    Guangzhou, 510275,
    P.R.~China}
\email{liunj@mail2.sysu.edu.cn}
	
\author{Jan Rozendaal}
\address{Institute of Mathematics, Polish Academy of Sciences\\
\'{S}niadeckich 8\\
00-656 Warsaw\\
Poland}
\email[corresponding author]{jrozendaal@impan.pl}

\author{Liang Song}
\address{Department of Mathematics,
    Sun Yat-sen University,
    Guangzhou, 510275,
    P.R.~China}
\email{songl@mail.sysu.edu.cn}



\thanks{This research was funded in part by the National Science Center, Poland, grant 2021/43/D/ST1/00667. N.~J. Liu is supported by China Postdoctoral Science Foundation (No. 2024M763732) and NNSF of China (No. 12501132). J. Rozendaal is partially supported by NCN grant UMO-2023/49/B/ST1/01961. L. Song is supported by  NNSF of China (No. 12471097).
}

\begin{document}

\begin{abstract}
We use the Hardy spaces for Fourier integral operators to obtain bounds for spherical maximal functions in $L^{p}(\mathbb{R}^{n})$, $n\geq2$, where the radii of the spheres are restricted to a compact interval in $(0,\infty)$. 
These bounds extend to general hypersurfaces with non-vanishing Gaussian curvature, to the complex spherical means, and to geodesic spheres on compact manifolds. We also obtain improved maximal function bounds and pointwise convergence statements for wave equations, both on $\mathbb{R}^{n}$ and on compact manifolds. The maximal function bounds are essentially sharp for all $p\in[1,2]\cup [\frac{2(n+1)}{n-1},\infty)$, for each such hypersurface, every complex spherical mean, and on every manifold.
\end{abstract}
	
\maketitle	


\section{Introduction}\label{sec:intro}

For $n\geq 2$, $f\in C(\Rn)$ and $x\in\Rn$, set
\begin{equation}\label{eq:defMsphere}
\Ma_{[1,2]}f(x):=\sup_{1\leq t\leq 2}\Big|\int_{S^{n-1}}f(x-ty)\ud\sigma(y)\Big|,
\end{equation}
where $\ud\sigma$ is the induced Lebesgue measure on the unit sphere $S^{n-1}$ in $\Rn$. In this article we will obtain $L^{p}(\Rn)$ bounds for $\Ma_{[1,2]}f$, and for related maximal functions, if $f$ is an element of a suitable Hardy space for Fourier integral operators. 

\subsection{Main results}\label{subsec:mainresults} For $p\in[1,\infty]$, set
\begin{equation}\label{eq:sp}
s(p):=\frac{n-1}{2}\Big|\frac{1}{2}-\frac{1}{p}\Big|
\end{equation}
and
\begin{equation}\label{eq:dp}
d(p):=\begin{cases}
s(p)-\frac{1}{p}&\text{if }\frac{2(n+1)}{n-1}\leq p\leq \infty,\\
0&\text{if }2\leq p<\frac{2(n+1)}{n-1},\\
s(p)&\text{if }1\leq p\leq 2.
\end{cases}
\end{equation}
For $s\in\R$, let $\Hps=\lb D\rb^{-s}\Hp$ be a Sobolev space over the Hardy space $\Hp$ for Fourier integral operators from Definition \ref{def:HpFIORn}. Then our main result for $\Ma_{[1,2]}$ is as follows.

\begin{theorem}\label{thm:maximalhypintro}
Let $p\in[1,\infty)$ and $s>d(p)+\frac{1}{p}-\frac{n-1}{2}$. Then there exists a $C\geq0$ such that
\begin{equation}\label{eq:maximalhyp}
\|\Ma_{[1,2]}f\|_{L^{p}(\Rn)}\leq C\|f\|_{\Hps}
\end{equation}
for all $f\in C(\Rn)\cap \Hps$. Conversely, if $p\in[1,2]\cup [\frac{2(n+1)}{n-1},\infty)$ and \eqref{eq:maximalhyp} holds for all $f\in\Sw(\Rn)$, then $s\geq d(p)+\frac{1}{p}-\frac{n-1}{2}$. 
\end{theorem}

In fact, \eqref{eq:maximalhyp} extends from $S^{n-1}$ to arbitrary compact hypersurfaces with non-vanishing Gaussian curvature, cf.~Theorem \ref{thm:maximalhyp1}, and a variable-coefficient version also holds; see Theorem \ref{thm:maximalhyp2}. A consequence of the latter is an extension of \eqref{eq:maximalhyp} to geodesic spheres on compact manifolds, Proposition \ref{prop:geodesicmax}. These results are all essentially sharp for $p\in[1,2]\cup[\frac{2(n+1)}{n-1},\infty)$, as is shown in Theorems \ref{thm:sharphyp1} and \ref{thm:sharphyp2} and Proposition \ref{prop:sharpman1}. See also Remark \ref{rem:pinfty} for the case where $p=\infty$.

Another extension of Theorem \ref{thm:maximalhypintro} arises by considering maximal functions associated with the complex spherical means from \cite{Stein76}. In Proposition \ref{prop:complexsphere} we show that \eqref{eq:maximalhyp} holds for $s>d(p)+\frac{1}{p}-\frac{n-1}{2}-\Real(\alpha)$ if one replaces the average in \eqref{eq:defMsphere} by the dilated complex spherical mean of order $\alpha\in\C$, with Theorem \ref{thm:maximalhypintro} corresponding to the case where $\alpha=0$. These bounds are essentially sharp for all $\alpha\in\C$ and $p\in[1,2]\cup[\frac{2(n+1)}{n-1},\infty)$, as follows from Proposition \ref{prop:sharpcomplex}.

We also obtain the following maximal function bounds and pointwise convergence statements for the Euclidean wave equation.

\begin{theorem}\label{thm:maximalFIOintro}
Let $p\in[1,\infty)$ and $s>d(p)+\frac{1}{p}$. Then there exists a $C\geq0$ such that
\begin{equation}\label{eq:maximalFIOintro}
\big\|\sup_{0\leq t\leq 1}|e^{it\sqrt{-\Delta}}f|\big\|_{L^{p}(\Rn)}\leq C\|f\|_{\Hps}
\end{equation}
for all $f\in\Hps$. Hence $e^{it\sqrt{-\Delta}}f(x)\to f(x)$ as $t\to 0$, for almost all $x\in\Rn$. Conversely, if $p\in[1,2]\cup [\frac{2(n+1)}{n-1},\infty)$ and \eqref{eq:maximalFIOintro} holds for all $f\in\Sw(\Rn)$, then $s\geq d(p)+\frac{1}{p}$. 
\end{theorem}

Theorem \ref{thm:maximalFIOintro} is a special case of Corollary \ref{cor:pointwiseFIO1} and Theorems \ref{thm:sharpFIO1} and \ref{thm:sharpFIO2}, with the latter two results containing the sharpness statement. A version of Theorem \ref{thm:maximalFIOintro} also holds on compact Riemannian manifolds $(M,g)$; see Proposition \ref{prop:maximalFIO3}. In this case one replaces the Euclidean Laplacian $\Delta$ by the Laplace--Beltrami operator $\Delta_{g}$, and $\Hps$ by the space $\HpsM$ from Definition \ref{def:HpFIOM}. On every such manifold, the resulting maximal function bounds are essentially sharp for $p\in[1,2]\cup[\frac{2(n+1)}{n-1},\infty)$, as is shown in Proposition \ref{prop:sharpman2}.

\subsection{Previous work}\label{subsec:previous}

Bounds for $\Ma_{[1,2]}$ arise naturally in the study of the $L^{p}$ boundedness of the full spherical maximal function $\Ma_{(0,\infty)}$, given by
\begin{equation}\label{eq:deffull}
\Ma_{(0,\infty)} f(x):=\sup_{0<t<\infty}\Big|\int_{S^{n-1}}f(x-ty)\ud\sigma(y)\Big|
\end{equation}
for $f\in C(\Rn)$ and $x\in\Rn$. In \cite{Stein76},
Stein proved the inequality
\begin{equation}\label{eq:fullLp}
\|\Ma_{(0,\infty)} f\|_{L^{p}(\Rn)}\lesssim \|f\|_{L^{p}(\Rn)}
\end{equation}
for $n\geq 3$ and $p>\frac{n}{n-1}$. It was later shown by Bourgain in \cite{Bourgain86b} 
that \eqref{eq:fullLp} also holds for $n=2$ and $p>2$. Since  \eqref{eq:fullLp} does not hold if $n=1$ and $p<\infty$, nor if $n\geq2$ and $p\leq \frac{n}{n-1}$, this completely settled the question of the $L^{p}$ boundedness of $\Ma_{(0,\infty)}$. Indeed, apart from the fact that the condition $p>\frac{n}{n-1}$ is sharp, for such $p$ one cannot improve \eqref{eq:fullLp} by considering a larger space than $L^{p}(\Rn)$ on the right-hand side of \eqref{eq:fullLp}. This follows from the fact that the spherical averages converge pointwise almost everywhere to a multiple of $f$ as $t\to0$.

On the other hand, one can work with a larger space on the right-hand side of \eqref{eq:fullLp} if one replaces $\Ma_{(0,\infty)} f$ by $\Ma_{[1,2]}f$. For example, inequalities of the form
\begin{equation}\label{eq:localimprove}
\|\Ma_{[1,2]}f\|_{L^{p}(\Rn)}\lesssim \|f\|_{W^{s,p}(\Rn)},
\end{equation}
for $n\geq3$, $p>\frac{n}{n-1}$ and suitable $s<0$, follow from the work of Sogge and Stein in \cite{Sogge-Stein90}. In fact, they obtained a version of \eqref{eq:fullLp} for more general hypersurfaces in $\Rn$, by combining bounds as in \eqref{eq:localimprove} with a scaling argument, and their techniques lead to the condition $s>\frac{1}{p}-(n-1)\min(\frac{1}{p},\frac{1}{p'})$ in \eqref{eq:localimprove} (see \cite[pp.~182-183]{Sogge-Stein90}). A similar approach was used in the work \cite{MoSeSo92} of Mockenhoupt, Seeger and Sogge, who relied on improved bounds for $\|\Ma_{[1,2]}f\|_{L^{p}(\R^{2})}$ to give a simplified proof of \eqref{eq:fullLp} for $n=2$ and $p>2$ (see also \cite{Sogge91,MoSeSo93} for more general hypersurfaces in $\R^{2}$).  Schlag \cite{Schlag97} obtained sharp $L^{p}\mapsto L^{q}$ bounds for $\Ma_{[1,2]}$, with $q>p$. These are stronger than $L^{p}\to L^{p}$ bounds and are related to \eqref{eq:localimprove} due to Sobolev embeddings (see \cite{Schlag-Sogge97,Lee03,Lacey19} as well). We also refer to e.g.~\cite{Roos-Seeger23} for results on spherical maximal functions in the case where the radii are restricted to a fractal subset of $[1,2]$. Finally, estimates as in \eqref{eq:localimprove} are connected to the frequently studied question of the $L^{p}$ boundedness of maximal functions associated with the complex spherical means  
(see e.g.~\cite{MoSeSo92,LiShSoYa25,MiYaZh17}). However, it is crucial for the present article to note that,  
for general $p$, it does not appear to be known what the optimal condition on $s$ is for \eqref{eq:localimprove} to hold. 

Bounds for spherical maximal functions are intimately connected to maximal function estimates for wave propagators.
This connection arises due to the fact that the Fourier transform of the measure $\ud\sigma$ in \eqref{eq:defM} can be expressed in terms of the Euclidean half-wave propagators. However, maximal function estimates for solutions to partial differential equations, and the associated pointwise convergence statements, are of independent interest. Such statements are particularly topical for the Schr\"{o}dinger equation, due to the recent (almost) resolution of Carleson's problem on the pointwise convergence of solutions to their initial data in $L^{2}(\Rn)$ (see \cite{Hickman23,Bourgain16,DuGuLi17,Du-Zhang19}). In the case of the wave equation, bounds of the form
\begin{equation}\label{eq:maximalFIOintro3}
\big\|\sup_{0\leq t\leq 1}|e^{it\sqrt{-\Delta}}f|\big\|_{L^{p}(\Rn)}\lesssim\|f\|_{W^{s,p}(\Rn)}
\end{equation}
are known to hold for $p\geq \frac{2(n+1)}{n-1}$ and $s>2s(p)$, as a corollary of  
\cite{Bourgain-Demeter15} (see \cite{BeHiSo20,BeHiSo21} for the analogous statement on compact manifolds). For $n=2$ one can obtain \eqref{eq:maximalFIOintro3} for $p>\frac{2n}{n-1}$ and $s>2s(p)$ by relying on \cite{GuWaZh20}, and similarly for compact Riemannian surfaces, using \cite{GaLiMiXi23}. These bounds are all essentially sharp in terms of the exponent $s$.
For $p=2$ and $s>1/2$, \eqref{eq:maximalFIOintro3} is due to Cowling \cite{Cowling83b}, and here the condition $s>1/2$ is necessary \cite{Walther99}. In particular, given that $\HT^{s,2}_{FIO}(\Rn)=W^{s,2}(\Rn)$, this corresponds to the case $p=2$ in Theorem \ref{thm:maximalFIOintro}. We also refer to \cite{Rogers-Villarroya08,ChLeLi26} for endpoint bounds in the setting where $p$ equals $2$ on the right-hand side, but not on the left-hand side, of \eqref{eq:maximalFIOintro3}.  However, for general $p$ it does not appear to be known what the optimal condition on $s$ is for \eqref{eq:maximalFIOintro3} to hold.

\subsection{Discussion of the results}\label{subsec:discussion}

In this article we obtain versions of \eqref{eq:localimprove} and \eqref{eq:maximalFIOintro3} where $W^{s,p}(\Rn)$ is replaced by a Hardy space for Fourier integral operators. In doing so we improve, for certain values of $p$, 
the known results regarding 
\eqref{eq:localimprove} and \eqref{eq:maximalFIOintro3},  
even in cases where the essentially optimal Sobolev exponent $s$ (in the $W^{s,p}$ scale) has been established. In other cases we complement the existing results.

The Hardy space $\Hp$ for Fourier integral operators was introduced for $p=1$ by Smith in \cite{Smith98a}, and his definition was extended to all $1\leq p\leq \infty$  by Hassell, Portal and the third author in \cite{HaPoRo20}. These spaces are invariant under suitable Fourier integral operators, and they satisfy the Sobolev embeddings
\begin{equation}\label{eq:Sobolevintro}
W^{s(p),p}(\Rn)\subseteq \Hp\subseteq W^{-s(p),p}(\Rn)
\end{equation}
for $1<p<\infty$, with the natural modifications involving the local Hardy space $\HT^{1}(\Rn)$ for $p=1$, and $\bmo(\Rn)$ for $p=\infty$. 

In particular, for $1<p<\infty$, Theorem \ref{thm:maximalhypintro} yields \eqref{eq:localimprove} with $s>d(p)+s(p)+\frac{1}{p}-\frac{n-1}{2}$, and Theorem \ref{thm:maximalFIOintro} yields \eqref{eq:maximalFIOintro3} with $s>d(p)+s(p)+\frac{1}{p}$. One has $\HT^{2}_{FIO}(\Rn)=L^{2}(\Rn)$, so for $p=2$ all our results apply to $W^{s,2}(\Rn)$. However, for $p\neq 2$ the corollaries of Theorems \ref{thm:maximalhypintro} and \ref{thm:maximalFIOintro} for $W^{s,p}(\Rn)$ are strictly weaker than the estimates in Theorems \ref{thm:maximalhypintro} and \ref{thm:maximalFIOintro} themselves, since 
$W^{s(p),p}(\Rn)$ is a proper subspace of $\Hp$. In fact, the Sobolev exponents in \eqref{eq:Sobolevintro} are sharp, and therefore the bounds in Theorems \ref{thm:maximalhypintro} and \ref{thm:maximalFIOintro} are significantly stronger than their corollaries regarding \eqref{eq:localimprove} and \eqref{eq:maximalFIOintro3}, at least for certain initial data. This is particularly relevant for $p\geq \frac{2(n+1)}{n-1}$, where one has $d(p)+s(p)+\frac{1}{p}=2s(p)$ and the resulting condition on $s$ in \eqref{eq:localimprove} and \eqref{eq:maximalFIOintro3} is essentially sharp, due to the mapping properties of the individual spherical means and half-wave propagators. Theorems \ref{thm:maximalhypintro} and \ref{thm:maximalFIOintro} show that \eqref{eq:localimprove} and \eqref{eq:maximalFIOintro3} can nonetheless be improved, by working with a different space of initial data. We also note that, for suitable Fourier integral operators $T$, one can replace $f$ by $Tf$ on the left-hand sides of \eqref{eq:maximalhyp} and \eqref{eq:maximalFIOintro}, whereas this is not possible when working directly with \eqref{eq:localimprove} and \eqref{eq:maximalFIOintro3}.

For $n=2$ and $p>\frac{2n}{n-1}$, one can use \cite{GuWaZh20} to obtain \eqref{eq:localimprove} and \eqref{eq:maximalFIOintro3} for $s>2s(p)-\frac{n-1}{2}$ and $s>2s(p)$, respectively. 
Due to the sharpness of the Sobolev exponents in \eqref{eq:Sobolevintro}, for $\frac{2n}{n-1}<p<\frac{2(n+1)}{n-1}$ these bounds do not follow from Theorems \ref{thm:maximalhypintro} and \ref{thm:maximalFIOintro}, 
nor does the converse implication hold. 
A similar comparison holds for the analogues of Theorems \ref{thm:maximalhypintro} and \ref{thm:maximalFIOintro} on compact manifolds, Propositions \ref{prop:geodesicmax} and \ref{prop:maximalFIO3}, and the results in \cite{GaLiMiXi23}. 

While the sharp $L^{p}\to L^{q}$ improving properties from \cite{Schlag97,Schlag-Sogge97,Lee03} only hold for $\frac{n}{n-1}<p\leq q$, Theorems \ref{thm:maximalhypintro} and \ref{thm:maximalFIOintro} yield sharp bounds for all $p\in[1,2]\cup [\frac{2(n+1)}{n-1},\infty)$. 
In fact, some of our estimates involving $\Hp$ are sharp even when the analogous bounds for $L^{p}(\Rn)$ can be improved (see e.g.~Remark \ref{rem:SSS}).

\subsection{Ideas behind the proofs}\label{subsec:proof}

Just as in \cite{Sogge-Stein90,MoSeSo92,Schlag-Sogge97} and in more recent work on the $L^{p}$ boundedness of complex spherical maximal functions and related operators (see e.g.~\cite{MiYaZh17,LiShSoYa25,BeRaSa19}), our spherical maximal function bounds follow from estimates for wave propagators.  
More precisely, to derive Theorem \ref{thm:maximalhypintro}, we make use of both fixed-time and local smoothing estimates for Fourier integral operators, as can be found in \cite{Smith98a,HaPoRo20,Rozendaal22b,LiRoSoYa24}. This approach, which relies on the same ideas as in \cite{Sogge-Stein90} (see also \cite{Sogge17,BeHiSo21}), leads to a version of Theorem \ref{thm:maximalFIOintro} for general Fourier integral operators, which in turn implies both Theorems \ref{thm:maximalhypintro} and \ref{thm:maximalFIOintro} as well as their various extensions. 

Although our approach to obtaining maximal function bounds is analogous to that in the literature, one of the main goals of this article is to prove estimates which are sharp in a variety of geometric settings, and a significant part of the article is dedicated to showing that our results are indeed essentially sharp. As in the case of the results on local smoothing in \cite{Rozendaal22b,LiRoSoYa24}, working with the Hardy spaces for Fourier integral operators allows one to obtain sharp estimates even where the corresponding sharp bounds for $L^{p}(\Rn)$ are not yet known, at the price of the additional work required to deal with the more complicated $\Hp$ norm.

In the case of Theorem \ref{thm:maximalFIOintro}, sharpness follows for $p\geq \frac{2(n+1)}{n-1}$ by combining  \eqref{eq:Sobolevintro} with the $L^{p}$ mapping properties of the individual operators, but in Theorem \ref{thm:sharpFIO1} we prove a more general lower bound for $s$ in \eqref{eq:maximalFIOintro} which does not hold in the $L^{p}$ scale (see Remark \ref{rem:SSS}).  To prove sharpness for $1\leq p\leq 2$ one has to produce additional focusing of order $1/p$, which we do in Theorem \ref{thm:sharpFIO2} using a Knapp example that cannot be used to prove sharpness of the corresponding estimate in the $L^{p}$ scale (see Remark \ref{rem:sharpLp}). In fact, Theorems \ref{thm:sharpFIO1} and \ref{thm:sharpFIO2} both concern more general Fourier integral operators, and the proofs rely on an improvement of a lemma from \cite{Rozendaal22b,LiRoSoYa24} which relates such operators to associated flow maps, Lemma \ref{lem:flow}. We do not know whether our results are sharp for $2<p<\frac{2(n+1)}{n-1}$ (see also Remark \ref{rem:sharpintermediate}).

Although the spherical averages can be expressed as a linear combination of terms involving the half-wave propagators, the contributions of these terms might, a priori, cancel each other out. To nonetheless derive sharpness of Theorem \ref{thm:maximalhypintro} from sharpness of maximal function bounds for half-wave propagators, in the proof of Theorem \ref{thm:sharphyp1} we microlocalize to regions of phase space where there is only a single such term. Moreover, due to the fact that our maximal function bounds for Fourier integral operators are sharp under very mild conditions, we can use FIO calculus to show that various extensions of Theorem \ref{thm:maximalhypintro} are also essentially sharp. 

Finally, given that the assumptions for the sharpness statements in Theorems \ref{thm:sharpFIO1} and \ref{thm:sharpFIO2} are of a local nature, it suffices to work in local coordinates when deriving sharpness of our results on manifolds. There one can use that our other results are sharp under very mild conditions. 

\subsection{Organization}

In Section \ref{sec:HpFIO} we define the Hardy spaces for Fourier integral operators on $\Rn$ and on compact manifolds, and we collect some of their properties, together with the bounds for Fourier integral operators which will be used to prove Theorems \ref{thm:maximalhypintro} and \ref{thm:maximalFIOintro}. Section \ref{sec:maximal} contains the statements and proofs of our results on maximal functions, with Section \ref{subsec:FIOmaximal} being dedicated to maximal function estimates for Fourier integral operators, Section \ref{subsec:hypersurf} to hypersurfaces in $\Rn$, Section \ref{subsec:complex} to the complex spherical averages, and Section \ref{subsec:manifold} to maximal functions on manifolds. Finally, in Section \ref{sec:sharpness} we show that the results in Section \ref{sec:maximal} are essentially sharp for suitable values of $p\in[1,\infty]$.

\subsection{Notation and terminology}

The natural numbers are $\N=\{1,2,\ldots\}$, and $\Z_{+}:=\N\cup\{0\}$. Throughout, we fix an $n\in\N$ with $n\geq 2$. 

For $\xi\in\Rn$ we write $\lb\xi\rb=(1+|\xi|^{2})^{1/2}$, and $\hat{\xi}=\xi/|\xi|$ if $\xi\neq0$. We use multi-index notation, where $\partial_{\xi}=(\partial_{\xi_{1}},\ldots,\partial_{\xi_{n}})$ and $\partial^{\alpha}_{\xi}=\partial^{\alpha_{1}}_{\xi_{1}}\ldots\partial^{\alpha_{n}}_{\xi_{n}}$
for $\xi=(\xi_{1},\ldots,\xi_{n})\in\Rn$ and $\alpha=(\alpha_{1},\ldots,\alpha_{n})\in\Z_{+}^{n}$. Moreover, $\partial_{x\eta}^{2}\Phi$ denotes the mixed Hessian of a function $\Phi$ of the variables $x$ and $\eta$.

The Fourier transform of a tempered distribution $f\in\Sw'(\Rn)$ is denoted by $\F f$ or $\widehat{f}$, and the Fourier multiplier with symbol $\ph$ is $\ph(D)$. The space of bounded linear operators 
between Banach spaces $X$ and $Y$ is $\La(X,Y)$, and $\La(X):=\La(X,X)$. 

We write $f(s)\lesssim g(s)$ to indicate that $f(s)\leq Cg(s)$ for all $s$ and a constant $C>0$ independent of $s$, and similarly for $f(s)\gtrsim g(s)$ and $g(s)\eqsim f(s)$.

\section{Hardy spaces for Fourier integral operators}\label{sec:HpFIO}

In this section we first define the Hardy spaces for Fourier integral operators on $\Rn$. We then collect some of their basic properties, together with bounds for suitable Fourier integral operators. In the final subsection we introduce the Hardy spaces for Fourier integral operators on compact manifolds.

\subsection{Definitions on $\Rn$}\label{subsec:definitionsRn}

For simplicity of notation, we write $\HT^{p}(\Rn):=L^{p}(\Rn)$ for $p\in(1,\infty)$. Moreover, $\HT^{1}(\Rn)$ is the local real Hardy space, with norm
\[
\|f\|_{\HT^{1}(\Rn)}:=\|q(D)f\|_{L^{1}(\Rn)}+\|(1-q)(D)f\|_{H^{1}(\Rn)}
\]
for $f\in\HT^{1}(\Rn)$. Here and throughout, $q\in C^{\infty}_{c}(\Rn)$ is such that $q(\xi)=1$ for $|\xi|\leq 2$, and $H^{1}(\Rn)$ is the classical real Hardy space. Finally, $\HT^{\infty}(\Rn):=\bmo(\Rn)$ is the dual of $\HT^{1}(\Rn)$, and we write $\HT^{s,p}(\Rn):=\lb D\rb^{-s}\HT^{p}(\Rn)$ for $p\in[1,\infty]$ and $s\in\R$.

Fix a collection $(\ph_{\w})_{\w\in S^{n-1}}\subseteq C^{\infty}(\Rn)$ with the following properties:
\begin{enumerate}
\item\label{it:phiproperties1} For all $\w\in S^{n-1}$ and $\xi\neq0$, one has $\ph_{\w}(\xi)=0$ if $|\xi|<\frac{1}{4}$ or $|\hat{\xi}-\w|>|\xi|^{-1/2}$.
\item\label{it:phiproperties2} For all $\alpha\in\Z_{+}^{n}$ and $\beta\in\Z_{+}$, there exists a $C_{\alpha,\beta}\geq0$ such that
\[
|(\w\cdot \partial_{\xi})^{\beta}\partial^{\alpha}_{\xi}\ph_{\w}(\xi)|\leq C_{\alpha,\beta}|\xi|^{\frac{n-1}{4}-\frac{|\alpha|}{2}-\beta}
\]
for all $\w\in S^{n-1}$ and $\xi\neq0$.
\item\label{it:phiproperties3}
The map $(\xi,\w)\mapsto \ph_{\w}(\xi)$ is measurable on $\Rn\times S^{n-1}$, and 
\[
\int_{S^{n-1}}\ph_{\w}(\xi)^{2}\ud\w=1
\]
for all $\xi\in\Rn$ with $|\xi|\geq 1$.
\end{enumerate}
For the construction of such a family, see \cite{Rozendaal21,HaPoRoYu23}. 

We can now define the Hardy spaces for Fourier integral operators.

\begin{definition}\label{def:HpFIORn}
For $p\in[1,\infty)$, the space $\Hp$ consists of all $f\in\Sw'(\Rn)$ such that $q(D)f\in L^{p}(\Rn)$, $\ph_{\w}(D)f\in \HT^{p}(\Rn)$ for almost all $\w\in S^{n-1}$, and $(\int_{S^{n-1}}\|\ph_{\w}(D)f\|_{\HT^{p}(\Rn)}^{p}\ud\w)^{1/p}<\infty$, endowed with the norm
\[
\|f\|_{\HT^{p}_{FIO}(\Rn)}:=\|q(D)f\|_{L^{p}(\Rn)}+\Big(\int_{S^{n-1}}\|\ph_{\w}(D)f\|_{\HT^{p}(\Rn)}^{p}\ud\w\Big)^{1/p}.
\]
Moreover, $\HT^{\infty}_{FIO}(\Rn):=(\HT^{1}_{FIO}(\Rn))^{*}$. For $p\in[1,\infty]$ and $s\in\R$, $\HT^{s,p}_{FIO}(\Rn)$ consists of all $f\in\Sw'(\Rn)$ such that $\lb D\rb^{s}f\in\HT^{p}_{FIO}(\Rn)$, endowed with the norm
\[
\|f\|_{\HT^{s,p}_{FIO}(\Rn)}:=\|\lb D\rb^{s}f\|_{\Hp}.
\]
\end{definition}

Originally, in \cite{Smith98a,HaPoRo20}, the Hardy spaces for FIOs were defined differently. It follows from \cite{Rozendaal21,FaLiRoSo23} that the original definition is equivalent to the one given here.

For $m\in\R$, we will work with the Kohn--Nirenberg symbol class $S^{m}(\R^{n+1}\times \R^{n})$, consisting of those $a\in C^{\infty}(\R^{n+1}\times\R^{n})$ such that
\begin{equation}\label{eq:Kohn-Nirenberg}
\sup\{\lb\eta\rb^{|\beta|-m}|\partial_{z}^{\alpha}\partial_{\eta}^{\beta}a(z,\eta)|\mid z\in\R^{n+1},\eta\in \R^{n}\}<\infty
\end{equation}
for all $\alpha\in\Z_{+}^{n+1}$ and $\beta\in \Z_{+}^{n}$. We write $S^{m}(\R^{n})$ for the subspace of symbols which do not depend on the $z$ variable.

\subsection{Results on $\Rn$}\label{subsec:results}

In this subsection we first collect a few basic properties of the Hardy spaces for Fourier integral operators, and then we state two results about the boundedness of suitable Fourier integral operators.

Firstly, the following extension of the Sobolev embeddings from \eqref{eq:Sobolevintro} holds:
\begin{equation}\label{eq:Sobolev}
\HT^{s+s(p),p}(\Rn)\subseteq \HT^{s,p}_{FIO}(\Rn)\subseteq \HT^{s-s(p),p}(\Rn),
\end{equation}
for all $p\in[1,\infty]$ and $s\in\R$, as follows from \cite[Theorem 7.4]{HaPoRo20}. Here $s(p)$ is as in \eqref{eq:sp}. If $p<\infty$, then the Schwartz functions lie dense in $\Hps$ (see \cite[Proposition 6.6]{HaPoRo20}). Although we will not make explicit use of the following properties in the present article, we also note that the Hardy spaces for FIOs form a complex interpolation scale, cf.~\cite[Proposition 6.7]{HaPoRo20}, and that $(\Hps)^{*}=\HT^{-s,p'}_{FIO}(\Rn)$ for all $p\in[1,\infty)$ and $s\in\R$, cf.~\cite[Proposition 6.8]{HaPoRo20}. By \cite{Smith98a,HaPoRo20,LiRoSoYa24}, $\HT^{s,1}_{FIO}(\Rn)$ has suitable molecular and atomic decompositions.

The following proposition is a slight generalization of \cite[Corollary 3.6 and Theorem 4.4]{Rozendaal22b}, to the case where the phase function is not defined globally and the symbol is time dependent. Recall that a subset $V\subseteq \Rn\setminus\{0\}$ is \emph{conic} if $\lambda \eta\in V$ for all $\eta\in V$ and $\lambda>0$, and recall the definition of the exponent $d(p)$ from \eqref{eq:dp}.

\begin{proposition}\label{prop:FIObounds1}
Let $V\subseteq \Rn\setminus\{0\}$ be an open conic subset, and let $\phi\in C^{\infty}(V)$ be real-valued and positively homogeneous of degree $1$. Let $m\in\R$ and $a\in S^{m}(\Rn)$ be such that $\supp(a)\subseteq V'\cup\{0\}$, for some conic $V'\subseteq V$ which is closed in $\Rn\setminus\{0\}$. Then, for all $p\in[1,\infty]$ and $s\in\R$, and for each compact interval $I\subseteq (0,\infty)$, the following statements hold.
\begin{enumerate}
\item\label{it:FIObounds11}  
One has $e^{it\phi(D)}a(tD):\HT^{s+m,p}_{FIO}(\Rn)\to \Hps$ for all $t>0$, and
\begin{equation}\label{eq:FIObounds1}
\sup_{t\in I}\|e^{it\phi(D)}a(tD)\|_{\La(\HT^{s+m,p}_{FIO}(\Rn),\Hps)}<\infty.
\end{equation}
\item\label{it:FIObounds12} Suppose, additionally, that $\rank\,\partial_{\eta\eta}^{2}\phi(\eta)=n-1$ for all $\eta\in V$. If $p\in(2,\infty)$ and $s>d(p)+m$, 
then there exists a $C\geq0$ such that
\begin{equation}\label{eq:FIObounds2}
\Big(\int_{I}\|e^{it\phi(D)}a(tD)f\|_{L^{p}(\Rn)}^{p}\ud t\Big)^{1/p}\leq C\|f\|_{\HT^{s,p}_{FIO}(\Rn)}
\end{equation}
for all $f\in \HT^{s,p}_{FIO}(\Rn)$.
\end{enumerate}
\end{proposition}
Note that $e^{it\phi(D)}a(D)$ is well-defined if we extend $\phi$ to all of $\Rn$. The operator then extends by adjoint action to $\Sw'(\Rn)$. A similar remark applies to \eqref{eq:Tstandard} below.
\begin{proof}
Let $\chi\in C^{\infty}(\Rn\setminus\{0\})$ be real-valued and positively homogeneous of degree zero, and such that $\supp(\chi)\subseteq V$ and $\chi\equiv1$ on $V'$. Then $\phi':=\phi\chi\in C^{\infty}(\Rn\setminus\{0\})$ is real-valued and  positively homogeneous of degree $1$. Moreover, $\supp(a(t\cdot))\subseteq V'\cup\{0\}$ for all $t>0$, since $V'$ is conic, and $\{a(t\cdot)\mid t\in I\}$ is a uniformly bounded collection in $S^{m}(\Rn)$. Hence \eqref{it:FIObounds11} follows by applying \cite[Lemma 3.5]{Rozendaal22b} to $e^{it\phi'(D)}a(tD)\lb D\rb^{-m}=e^{it\phi(D)}a(tD)\lb D\rb^{-m}$, for $t\in I$. 

On the other hand, the proof of \eqref{it:FIObounds12} is analogous to that of \cite[Theorem 4.4]{Rozendaal22b}, relying in particular on the $\ell^{p}$ decoupling inequality from \cite{Bourgain-Demeter17}, for hypersurfaces with non-vanishing Gaussian curvature. Note that the condition that $\rank\,\partial_{\eta\eta}^{2}\phi(\eta)=n-1$ for all $\eta\in V$ is equivalent to $\{\eta\in V\mid \phi(\eta)=1\}$ having non-vanishing Gaussian curvature. 
\end{proof}

The following proposition is a variable-coefficient version of Proposition \ref{prop:FIObounds1}, involving operators with symbols that have compact spatial support. A subset $V\subseteq \R^{n+1}\times(\Rn\setminus\{0\})$ is \emph{conic} if $(z,\lambda \eta)\in V$ whenever $(z,\eta)\in V$ and $\lambda>0$. 

\begin{proposition}\label{prop:FIObounds2}
Let $V\subseteq \R^{n+1}\times(\R^{n}\setminus\{0\})$ be an open conic subset, and let $\Phi\in C^{\infty}(V)$ be real-valued and positively homogenous of degree $1$ in the fiber variable. Let $m\in\R$ and $a\in S^{m}(\R^{n+1}\times\Rn)$ be such that $(z,\eta)\in V$ whenever $(z,\eta)\in\supp(a)$ and $\eta\neq0$. Suppose that $\rank\,\partial_{x\eta}^{2}\Phi(z_{0},\eta_{0})=n$ for all $(z_{0},\eta_{0})\in\supp(a)$ with $\eta_{0}\neq 0$, where we write $z=(x,t)\in\Rn\times\R$, and that there exists a compact set $K\subseteq\R^{n+1}$ such that $a(z,\eta)=0$ for all $(z,\eta)\in\R^{n+1}\times\Rn$ with $z\notin K$. Set
\begin{equation}\label{eq:Tstandard}
Tf(z):=\int_{\Rn}e^{i\Phi(z,\eta)}a(z,\eta)\wh{f}(\eta)\ud\eta
\end{equation}
for $f\in\Sw(\Rn)$ and $z\in\R^{n+1}$, and $T_{t}f(x):=Tf(x,t)$. Let $p\in[1,\infty]$ and $s\in\R$. Then the following statements hold.
\begin{enumerate}
\item\label{it:FIObounds21} One has $T_{t}:\HT^{s+m,p}_{FIO}(\Rn)\to \HT^{s,p}_{FIO}(\Rn)$ for all $t\in\R$, and
\begin{equation}\label{eq:FIOunif}
\sup_{t\in\R}\|T_{t}\|_{\La(\HT^{s+m,p}_{FIO}(\Rn),\HT^{s,p}_{FIO}(\Rn))}<\infty.
\end{equation}
\item\label{it:FIObounds22}
Suppose, additionally, that for all $(z_{0},\eta_{0})\in\supp(a)$ with $\eta_{0}\neq 0$, one has
\begin{equation}\label{eq:Gaussmap}
\rank\,\partial_{\eta\eta}^{2}(\partial_{z}\Phi(z_{0},\eta)\cdot G(z_{0},\eta_{0}))|_{\eta=\eta_{0}}=n-1,
\end{equation}
where $G(z_{0},\eta_{0})$ is the generalized Gauss map. Then, for all $p\in(2,\infty)$ and $s>d(p)+m$, one has 
\[
T:\HT^{s,p}_{FIO}(\Rn)\to L^{p}(\R^{n+1}).
\]
\end{enumerate}
\end{proposition}

\begin{proof}
The result is contained in \cite{LiRoSoYa24}. More precisely, \eqref{it:FIObounds21} is \cite[Corollary 2.4 and Remark 2.6]{LiRoSoYa24}, and \eqref{it:FIObounds22} can be deduced from \cite[Corollary 5.4]{LiRoSoYa24} in the same way as in the proof of \cite[Theorem 5.7]{LiRoSoYa24}.
\end{proof}

\begin{remark}\label{rem:cinematic}
Recall that, for $(z_{0},\eta_{0})\in\supp(a)$ with $\eta_{0}\neq 0$, the generalized Gauss map is $G(z_{0},\eta_{0})=G_{0}(z_{0},\eta_{0})/|G_{0}(z_{0},\eta_{0})|$, where $G_{0}(z_{0},\eta_{0}):=\wedge_{i=1}^{n}\partial_{\eta_{i}}\partial_{z}\Phi(z_{0},\eta_{0})$. In other words, \eqref{eq:Gaussmap} says that the unit normal vectors $\pm\theta$ to the conical hypersurface $\{\partial_{z}\Phi(z_{0},\la \eta)\mid (z_{0},\eta)\in \supp(a),\la>0\}\subseteq\R^{n+1}$ at $\partial_{z}\Phi(z_{0},\eta_{0})$ satisfy 
\[
\rank\,\partial^{2}_{\eta\eta}(\partial_{z}\Phi(z_{0},\eta)\cdot\theta)
|_{\eta=\eta_{0}}=n-1.
\]
This is in turn equivalent to $\{\partial_{z}\Phi(z_{0},\la \eta)\mid (z_{0},\eta)\in \supp(a),\la>0\}\subseteq\R^{n+1}$ having $n-1$ non-vanishing principal curvatures, which, in the terminology of \cite{Sogge91}, means that the canonical relation of $T$ satisfies the \emph{cinematic curvature condition}.

Note also that, if \eqref{eq:Gaussmap} holds, then one cannot have $\partial_{t}\Phi(x_{0},t,\eta_{0})|_{t=t_{0}}=0$ for all $(z_{0},\eta_{0})$ in an open subset of $\supp(a)$, where we write $z_{0}=(x_{0},t_{0})$. Indeed, otherwise an open subset of the conical hypersurface $\{\partial_{z}\Phi(z_{0},\la \eta)\mid (z_{0},\eta)\in \supp(a),\la>0\}$ would be contained in the plane where the last coordinate vanishes, which would prohibit it from having $n-1$ nonvanishing principal curvatures.
\end{remark}

We will also use frequently that $\Hps$ is invariant under changes of coordinates, for all $1\leq p\leq \infty$ and $s\in\R$, as follows from \cite[Theorem 6.10]{HaPoRo20} (see also \cite[Proposition 3.3]{Rozendaal22b}).

\subsection{Spaces on manifolds}\label{subsec:spacesman}

Let $(M,g)$ be a (smooth) $n$-dimensional compact Riemannian manifold without boundary, and $K$ a finite collection of coordinate charts, $\ka:U_{\ka}\to \ka(U_{k})\subseteq\Rn$, such that $\cup_{\ka\in K}U_{\ka}=M$. Let $(\psi_{\ka})_{\ka\in K}\subseteq C^{\infty}(M)$ be such that 
\begin{equation}\label{eq:partitionunit}
\sum_{\ka\in K}\psi_{\ka}^{2}(x)=1
\end{equation}
for all $x\in M$, and such that $\supp(\psi_{\ka})\subseteq U_{\ka}$ for each $\ka\in K$. 

We denote the collection of smooth half densities on $M$ by $\Da(M,\Omega_{1/2})$, and its dual, the space of distributional half densities, by $\Da'(M,\Omega_{1/2})$. We recall that each $f\in\Da(M,\Omega_{1/2})$ is of the form $f=\tilde{f}\sqrt{d V_{g}}$ for some $\tilde{f}\in C^{\infty}(M)$, and each $f\in\Da'(M,\Omega_{1/2})$ is of the form $f=\tilde{f}(d V_{g})^{-1/2}$ for some linear functional $\tilde{f}$ on $C^{\infty}(M)$. Here $d V_{g}$ is the Riemannian density on $M$. 

Finally, for $\ka\in K$, we write $\ka^{*}$ for pullback by $\ka$, acting on either smooth or distributional half densities, and $\ka_{*}:=(\ka^{*})^{-1}$.

\begin{definition}\label{def:HpFIOM}
Let $p\in[1,\infty]$ and $s\in\R$. Then $\HpsM$ consists of those $f\in\Da'(M,\Omega_{1/2})$ such that $\ka_{*}(\psi_{k}f)\in \Hps$ for all $\ka\in K$, endowed with the norm $\|f\|_{\HpsM}:=\|(\ka_{*}(\psi_{\ka}f))_{\ka\in K}\|_{\ell^{p}(K;\Hps)}$.
\end{definition}

More explicitly,
\begin{equation}\label{eq:HpFIOnormM1}
\|f\|_{\HpsM}:=\Big(\sum_{\ka\in K}\|\ka_{*}(\psi_{\ka}f)\|_{\Hps}^{p}\Big)^{1/p}
\end{equation}
for $p<\infty$, and 
\begin{equation}\label{eq:HpFIOnormM2}
\|f\|_{\HT^{s,\infty}_{FIO}(M)}:=\max_{\ka\in K}\|\ka_{*}(\psi_{\ka}f)\|_{\HT^{s,\infty}_{FIO}(\Rn)}. 
\end{equation}

\begin{remark}\label{rem:defmanindependent}
Up to norm equivalence, the definition of $\HpsM$ is independent of the choice of $K$ and $(\psi_{\ka})_{\ka\in K}$ as bove. This follows from a simple calculation, using \eqref{eq:partitionunit} and the invariance of $\Hps$ under coordinate changes and multiplication by smooth cutoffs (see also \cite[Proposition 4.6]{LiRoSoYa24}). 
\end{remark}

\begin{remark}\label{rem:defman}
Since $K$ is finite, one could equally well replace the $\ell^{p}$ norms in \eqref{eq:HpFIOnormM1} and \eqref{eq:HpFIOnormM2} by an $\ell^{q}$ norm for any $1\leq q\leq\infty$. On the other hand, one can define $\HpsM$ on general Riemannian manifolds with bounded geometry, a class which includes both compact manifolds and $\Rn$, and for this larger class the $\ell^{p}$ norm in \eqref{eq:HpFIOnormM1} is natural, as is already evident on $\Rn$ (see \cite[Theorem 3.6]{LiRoSoYa24}).

In \cite{LiRoSoYa24}, $\HpsM$ was defined in terms of collections of geodesic normal coordinate charts satisfying a uniformity condition on the overlap of their domains. For finite $K$, the latter condition is automatically satisfied. Moreover, although one could choose $K$ to consist of geodesic normal coordinate charts, this assumption is not needed on compact manifolds, cf.~Remark \ref{rem:defmanindependent}. 
\end{remark}

The basic properties of $\Hps$ carry over to $\HpsM$, for all $p\in[1,\infty]$ and $s\in\R$. For example, $\HpsM$ is invariant under Fourier integral operators of order zero associated with a local canonical graph, cf.~\cite[Theorem 4.16]{LiRoSoYa24}, and the Sobolev embeddings
\begin{equation}\label{eq:SobolevM}
\HT^{s+s(p),p}(M)\subseteq \HT^{s,p}_{FIO}(M)\subseteq \HT^{s-s(p),p}(M)
\end{equation}
hold, by \cite[Theorem 4.15]{LiRoSoYa24}. Here $\HT^{s,p}(M)$ is defined in the same way as $\HpsM$, in local coordinates (in particular, \eqref{eq:SobolevM} then follows directly from \eqref{eq:Sobolev}). We also note that $\Da(M,\Omega_{1/2})$ lies dense in $\HpsM$  if $p<\infty$. 

\section{Maximal theorems}\label{sec:maximal}

In this section we prove various results about maximal functions, including in particular Theorems \ref{thm:maximalhypintro} and \ref{thm:maximalFIOintro}. We first prove two results concerning maximal functions for Fourier integral operators, from which we then derive statements about maximal functions associated with hypersurfaces in $\Rn$, complex spherical means and geodesic spheres on manifolds.

\subsection{Fourier integral operators}\label{subsec:FIOmaximal}

To deal with the low-frequency components of the operators that appear in this subsection, we will use the following quantified version of \cite[Lemma 1.17]{DosSantosFerreira-Staubach14}.

\begin{lemma}\label{lem:lowfreq}
There exists a $C\geq0$ such that the following holds. Let $b\in C^{\infty}(\Rn\times (\Rn\setminus\{0\}))$ be such that $b(x,\eta)=0$ for all $(x,\eta)\in\R^{2n}$ with $|\eta|>1$, and such that
\[
C_{\alpha}:=\sup\{|\eta|^{\max(0,-1+|\alpha|)}|\partial_{\eta}^{\alpha}b(x,\eta)|\mid (x,\eta)\in\Rn\times(\Rn\setminus\{0\})\}<\infty
\]
for every $\alpha\in\Z_{+}^{n}$ with $|\alpha|\leq n+1$. Set $C_{b}:=\max\{C_{\alpha}\mid \alpha\in\Z_{+}^{n},|\alpha|\leq n+1\}$. Then 
\[
\Big|\int_{\Rn}e^{i(x-y)\cdot\eta}b(x,\eta)\ud\eta\Big|\leq CC_{b}\lb x-y\rb^{-(n+1)}(1+\log(1+|x-y|))
\]
for all $x,y\in\Rn$. 
\end{lemma}
\begin{proof}
The proof is as in \cite[Lemma 1.17]{DosSantosFerreira-Staubach14}. For the convenience of the reader, we include the argument here.

Due to the support condition on $b$, it suffices to consider $x,y\in\Rn$ with $|x-y|\geq1$. Set 
\[
\beta(x,y,\eta):=(-i|x-y|^{-1}(x-y)\cdot\partial_{\eta})^{n}b(x,\eta)
\]
for $\eta\in\Rn\setminus\{0\}$, and let $\chi\in C^{\infty}_{c}(\Rn)$ be such that $\chi(\eta)=1$ if $|\eta|\leq 1$, and $\chi(\eta)=0$ if $|\eta|\geq 2$. Integration by parts then allows us to write
\begin{align*}
&|x-y|^{n}\int_{\Rn}e^{i(x-y)\cdot\eta}b(x,\eta)\ud\eta=\int_{\Rn}e^{i(x-y)\cdot\eta}\beta(x,y,\eta)\ud\eta\\
&=\int_{\Rn}e^{i(x-y)\cdot\eta}\chi(\eta/\veps)\beta(x,y,\eta)\ud\eta+\int_{\Rn}e^{i(x-y)\cdot\eta}(1-\chi(\eta/\veps))\beta(x,y,\eta)\ud\eta,
\end{align*}
for a given $0<\veps<1$.

By integrating in spherical coordinates, one sees that the first term on the second line above is bounded in absolute value by a constant multiple of $C_{b}\veps$. On the other hand, integrating by parts one more time shows that the second term equals 
\begin{align*}
&-|x-y|^{-1}\int_{\Rn}e^{i(x-y)\cdot\eta}\veps^{-1}(-i|x-y|^{-1}(x-y)\cdot\partial_{\eta}\chi)(\eta/\veps)\beta(x,y,\eta)\ud\eta\\
&+|x-y|^{-1}\int_{\Rn}e^{i(x-y)\cdot\eta}(1-\chi(\eta/\veps))(-i|x-y|^{-1}(x-y)\cdot\partial_{\eta}\beta(x,y,\eta))\ud\eta.
\end{align*}
Then integration in spherical coordinates shows that the term on the first line is bounded in absolute value by a constant multiple of $C_{b}|x-y|^{-1}$, and that the second term is bounded in absolute value by a multiple of $C_{b}|x-y|^{-1}\log\veps^{-1}$. Finally, setting $\veps=|x-y|^{-1}$ concludes the proof.
\end{proof}

We are now ready to prove a theorem about maximal functions for the Fourier integral operators from Proposition \ref{prop:FIObounds1}. 

\begin{theorem}\label{thm:maximalFIO1}
Let $V\subseteq \Rn\setminus\{0\}$ be an open conic subset, and let $\phi\in C^{\infty}(V)$ be real-valued, positively homogeneous of degree $1$, and such that $\rank\,\partial_{\eta\eta}^{2}\phi(\eta)=n-1$ for all $\eta\in V$. Let $m\in\R$ and $a\in S^{m}(\Rn)$ be such that $\supp(a)\subseteq V'\cup\{0\}$, for some conic $V'\subseteq V$ which is closed in $\Rn\setminus\{0\}$. Then, for all $p\in[1,\infty)$ and $s>d(p)+\frac{1}{p}+m$, and for each compact interval $I\subseteq(0,\infty)$, there exists a $C\geq0$ such that
\begin{equation}\label{eq:maximalFIO1}
\big\|\sup_{t\in I}|e^{it\phi(D)}a(tD)f|\big\|_{L^{p}(\Rn)}\leq C\|f\|_{\Hps}
\end{equation}
for all $f\in\Hps$.
\end{theorem}
By Theorems \ref{thm:sharpFIO1} and \ref{thm:sharpFIO2}, \eqref{eq:maximalFIO1} is essentially sharp for all $p\in[1,2]\cup [\frac{2(n+1)}{n-1},\infty)$ if there exist a non-empty conic subset $V''\subseteq V'$ and $c',C'>0$ such that $|a(\eta)|\geq c'|\eta|^{m}$ for all $\eta\in V''$ with $|\eta|\geq C'$.
\begin{proof}
First consider $f\in\Sw(\Rn)$. Then 
$(x,t)\mapsto e^{it\phi(D)}a(tD)f(x)$ is a continuous function on $\Rn\times I$. Hence $\sup_{t\in I}|e^{it\phi(D)}a(tD)f(x)|$ is a lower semi-continuous function of $x$, and as such measurable. This implies that the left-hand side of \eqref{eq:maximalFIO1} is well-defined.

\subsubsection{The low frequencies}

Now, to prove \eqref{eq:maximalFIO1}, we first deal with the low-frequency component of $a$. Let $r\in C^{\infty}_{c}(\Rn)$ be such that $r(\eta)=1$ if $|\eta|\leq 1/2$, and $r(\eta)=0$ if $|\eta|>1$. Note that $e^{it\phi(D)}a(tD)r(D)$ has kernel $K_{t}$ given by
\[
K_{t}(x,y)=\frac{1}{(2\pi)^{n}}\int_{\Rn}e^{i((x-y)\cdot\eta}e^{it\phi(\eta)}a(t\eta)r(\eta)\ud\eta,
\]
for $t\in I$ and $x,y\in\Rn$. Moreover, $\eta\mapsto e^{it\phi(\eta)}a(t\eta)r(\eta)$ is smooth on $\R^{n}\setminus\{0\}$, and
\[
\sup\big\{|\eta|^{\max(0,-1+|\alpha|)}\partial_{\eta}^{\alpha}\big(e^{it\phi(\eta)}a(t\eta)r(\eta)\big)\,\big|\, t\in I,\eta\in \Rn\setminus\{0\}\big\}<\infty
\]
for all $\alpha\in\Z_{+}^{n}$. Hence 
\[
\sup\{\lb x-y\rb^{n+1/2}|K_{t}(x,y)|\mid x,y\in\Rn,t\in I\}<\infty,
\]
by Lemma \ref{lem:lowfreq}.

Next, let $r'\in C^{\infty}_{c}(\Rn)$ be such that $r(D)=r(D)r'(D)$. Since $r'(D)$ is smoothing, Young's inequality and \eqref{eq:Sobolev} yield
\begin{align*}
&\big\|\sup_{t\in I}|e^{it\phi(D)}a(tD)r(D)f|\big\|_{L^{p}(\Rn)}^{p}=\int_{\Rn}\sup_{t\in I}\Big|\int_{\Rn}K_{t}(x,y)(r'(D)f)(y)\ud y\Big|^{p}\ud x\\
&\lesssim \int_{\Rn}\Big(\int_{\Rn}\lb x-y\rb^{-n-1/2}|(r'(D)f)(y)|\ud y\Big)^{p}\ud x\lesssim \|r'(D)f\|_{L^{p}(\Rn)}^{p}\lesssim \|f\|_{\Hps}^{p}.
\end{align*}
 In the remainder, it thus suffices to show that
\[
\big\|\sup_{t\in I}|e^{it\phi(D)}a(tD)f|\big\|_{L^{p}(\Rn)}\lesssim \|f\|_{\Hps}
\]
under the additional assumption that $\wh{f}(\xi)=0$ if $|\xi|\leq 1/2$.

\subsubsection{The high frequencies}

Note that
\begin{equation}\label{eq:mainproof0}
\begin{aligned}
\sup_{t\in I}\|e^{it\phi(D)}a(tD)f\|_{L^{p}(\Rn)}&\lesssim \sup_{t\in I}\|e^{it\phi(D)}a(tD)f\|_{\HT^{s(p),p}_{FIO}(\Rn)}\\
&\lesssim \|f\|_{\HT^{s(p)+m,p}_{FIO}(\Rn)}.
\end{aligned}
\end{equation}
Here we used the embeddings $\HT^{s(p),p}_{FIO}(\Rn)\subseteq \HT^{p}(\Rn)\subseteq L^{p}(\Rn)$ for the first inequality, and \eqref{eq:FIObounds1} for the second one. Moreover, for each $\veps>0$, one has
\begin{equation}\label{eq:mainproof1}
\Big(\int_{I}\|e^{ir\phi(D)}a(rD)f\|_{L^{p}(\Rn)}^{p}\ud r\Big)^{1/p}\lesssim \|f\|_{\HT^{d(p)+m+\veps,p}_{FIO}(\Rn)},
\end{equation}
by \eqref{eq:FIObounds1} and \eqref{eq:FIObounds2}. The same arguments yield
\begin{equation}\label{eq:mainproof2}
\Big(\int_{I}\|\partial_{r}(e^{ir\phi(D)}a(rD)f)\|_{L^{p}(\Rn)}^{p}\ud r\Big)^{1/p}\lesssim \|f\|_{\HT^{d(p)+m+1+\veps,p}_{FIO}(\Rn)},
\end{equation}
since $\wh{f}(\xi)=0$ for all $\xi\in\Rn$ with $|\xi|\leq 1/2$. 

Now, if $(\psi_{k})_{k=0}^{\infty}\subseteq C^{\infty}_{c}(\Rn)$ is a standard Littlewood--Paley decomposition, then the collection $(\psi_{k}(D))_{k=0}^{\infty}\subseteq \La(\Hps)$ is uniformly bounded, given that the convolution operators $\psi_{k}(D)$ have kernels that are uniformly in $L^{1}(\Rn)$. 
Hence, by standard arguments and because $\wh{f}(\xi)=0$ if $|\xi|\leq 1/2$, it suffices to show that, for all $\veps>0$ and $k\geq0$, one has
\[
\big\|\sup_{t\in I}|e^{it\phi(D)}a(tD)f|\big\|_{L^{p}(\Rn)}\lesssim 2^{k(d(p)+\frac{1}{p}+m+\veps)}\|f\|_{\Hp}
\]
if $\supp(\wh{f}\,)\subseteq\{\xi\in\Rn\mid 2^{k-1}\leq |\xi|\leq 2^{k+1}\}$. 

Fix such $\veps$ and $k$, and let $t_{0}$ be the left endpoint of $I$. Then one can use the fundamental theorem of calculus and H\"{o}lder's inequality to write
\begin{align*}
&|e^{it\phi(D)}a(tD)f(x)|^{p}= |e^{it_{0}\phi(D)}a(t_{0}D)f(x)|^{p}+\int_{t_{0}}^{t}\partial_{r}|e^{ir\phi(D)}a(rD)f(x)|^{p}\ud r\\
&\leq |e^{it_{0}\phi(D)}a(t_{0}D)f(x)|^{p}\\
&+p\Big(\int_{I}|e^{ir\phi(D)}a(rD)f(x)|^{p}\ud r\Big)^{1/p'}\Big(\int_{I}|\partial_{r}(e^{ir\phi(D)}a(rD)f)(x)|^{p}\ud r\Big)^{1/p}
\end{align*}
for all $t\in I$ and $x\in\Rn$. Hence, after integrating in $x$ and applying H\"{o}lder's inequality once again, one can use \eqref{eq:mainproof0}, \eqref{eq:mainproof1} and \eqref{eq:mainproof2} to write
\begin{align*}
&\big\|\sup_{t\in I}|e^{it\phi(D)}a(tD)f|\big\|_{L^{p}(\Rn)}^{p}\\
&\leq \|e^{it_{0}\phi(D)}a(t_{0}D)f\|_{L^{p}(\Rn)}^{p}\\
&+p\Big(\int_{\Rn}\int_{I}|e^{ir\phi(D)}a(rD)f(x)|^{p}\ud r\ud x\!\Big)^{1/p'}\!\Big(\int_{\Rn}\int_{I}|\partial_{r}(e^{ir\phi(D)}a(rD)f)(x)|^{p}\ud r\ud x\!\Big)^{1/p}\\
&\lesssim \|f\|_{\HT^{s(p)+m,p}_{FIO}(\Rn)}^{p}+\|f\|_{\HT^{d(p)+m+\veps,p}_{FIO}(\Rn)}^{p/p'}\|f\|_{\HT^{d(p)+m+1+\veps,p}_{FIO}(\Rn)}\\
&\eqsim\big(2^{kp(s(p)+m)}+2^{k(d(p)+m+\veps)\frac{p}{p'}}2^{k(d(p)+m+1+\veps)}\big)\|f\|_{\Hp}^{p}\\
&\eqsim 2^{kp(d(p)+\frac{1}{p}+m+\veps)}\|f\|_{\Hp}^{p},
\end{align*}
as required.

\subsubsection{Conclusion of the proof}

 It remains to extend \eqref{eq:maximalFIO1} to general $f\in\Hps$. Here it is relevant to note that, a priori, it is not even clear why the left-hand side of \eqref{eq:maximalFIO1} is well defined for such $f$. 
Indeed, initially $e^{it\phi(D)}a(tD)f$ is defined as a tempered distribution for each $t\in I$, but Proposition \ref{prop:FIObounds1} shows that it can be identified with an element of $L^{p}(\Rn)$. We will  prove that there exists a set $U\subseteq\Rn$ of measure zero with the property that, for every $t\in I$, there exists a representative $F(\cdot,t)$ of $e^{it\phi(D)}a(tD)f$ in $L^{p}(\Rn)$ such that the function $x\mapsto \sup_{t\in I}|F(x,t)|\ind_{\Rn\setminus U}(x)$ is measurable and satisfies $\|\sup_{t\in I}|F(\cdot,t)|\|_{L^{p}(\Rn)}\lesssim\|f\|_{\Hps}$.

There exists a sequence $(f_{k})_{k=0}^{\infty}\subseteq\Sw(\Rn)$ converging to $f$ in $\Hps$. Set $F_{k}(x,t):=e^{it\phi(D)}a(tD)f_{k}(x)$ for $k\geq0$ and 
%
 $(x,t)\in\Rn\times I$. 
By \eqref{eq:maximalFIO1}, $(F_{k})_{k=0}^{\infty}$ is a Cauchy sequence in $L^{p}(\Rn;C(I))$, so it has a limit $F\in L^{p}(\Rn;C(I))$ satisfying
\begin{equation}\label{eq:Fbound}
\big\|\sup_{t\in I}|F(\cdot,t)|\big\|_{L^{p}(\Rn)}\lesssim\|f\|_{\Hps},
\end{equation}
and 
\[
\big\|\sup_{t\in I}|F_k(\cdot,t)-F(\cdot,t)|\big\|_{L^{p}(\Rn)}\lesssim\|f_k-f\|_{\Hps}.
\]
In particular, due to the latter inequality, we may also suppose that there exists a set $U\subseteq\Rn$ of measure zero such that, for all $x\notin U$, one has $F_{k}(x,t)\to F(x,t)$, uniformly in $t\in I$. 
 
On the other hand, by Proposition \ref{prop:FIObounds1} \eqref{it:FIObounds11}, for each $t\in I$ one has $F_{k}(\cdot,t)\to e^{it\phi(D)}a(tD)f(\cdot)$ in $L^{p}(\Rn)$ as $k\to\infty$. So there exists a set $U_{t}\subseteq\Rn$ of measure zero, and a subsequence $(F_{m_{k,t}})_{k=0}^{\infty}$ of $(F_{k})_{k=0}^{\infty}$, such that $F_{m_{k,t}}(x,t)\to e^{it\phi(D)}a(tD)f(x)$ for $x\notin U_{t}$. It follows that $e^{it\phi(D)}a(tD)f(x)=F(x,t)$ for $x\notin U\cup U_{t}$. Since $e^{it\phi(D)}a(tD)f$ 
is only determined up to sets of measure zero, 
$F(\cdot,t)$ is in fact a representative 
of $e^{it\phi(D)}a(tD)f$ in $L^{p}(\Rn)$. Now the proof is concluded by applying \eqref{eq:Fbound}.
\end{proof}

The following corollary contains Theorem \ref{thm:maximalFIOintro} as a special case.

\begin{corollary}\label{cor:pointwiseFIO1}
Let $V\subseteq \Rn\setminus\{0\}$, $\phi\in C^{\infty}(V)$, $m\in\R$ and $a\in S^{m}(\Rn)$ be as in Theorem \ref{thm:maximalFIO1}, and let $p\in[1,\infty)$ and $s>d(p)+\frac{1}{p}+m$. Then, for all $t_{0}>0$ and $f\in\Hps$, and for almost all $x\in\Rn$, one has
\begin{equation}\label{eq:pointwise1}
e^{it\phi(D)}a(tD)f(x)\to e^{it_{0}\phi(D)}a(t_{0}D)f(x)
\end{equation}
as $t\to t_{0}$. Moreover, if $V=\Rn\setminus\{0\}$ and $a\equiv 1$, then \eqref{eq:maximalFIO1} holds for any bounded interval $I\subseteq\R$, and \eqref{eq:pointwise1} holds for all $t_{0}\in\R$.
\end{corollary}
\begin{proof}
First note that \eqref{eq:pointwise1} holds for all $x\in\Rn$ if $f\in\Sw(\Rn)$. Since the Schwartz functions lie dense in $\Hps$ if $p<\infty$, and because $d(p)+\frac{1}{p}\geq s(p)$, one can apply \eqref{eq:maximalFIO1} in a standard manner to obtain \eqref{eq:pointwise1}.

Next, suppose that $V=\Rn\setminus\{0\}$ and that $a\equiv 1$, and let $I=[t_{1},t_{2}]\subseteq\R$ be a bounded interval. For $t_{3}<t_{1}$, one can use \eqref{eq:maximalFIO1} and Proposition \ref{prop:FIObounds1} \eqref{it:FIObounds11} to write
\begin{align*}
\big\|\sup_{t\in I}|e^{it\phi(D)}f|\big\|_{L^{p}(\Rn)}&=\big\|\sup_{t_{1}-t_{3}\leq t\leq t_{2}-t_{3}}|e^{it\phi(D)}(e^{it_{3}\phi(D)}f)|\big\|_{L^{p}(\Rn)}\\
&\lesssim \|e^{it_{3}\phi(D)}f\|_{\Hps}\eqsim \|f\|_{\Hps}
\end{align*}
for all $f\in\Hps$. Now the extension of \eqref{eq:pointwise1} to $t_{0}\in\R$ follows as before.
\end{proof}

\begin{remark}\label{rem:pinfty}
Consider the case where $p=\infty$ in Theorem \ref{thm:maximalFIO1}. For all $s>d(\infty)+m$ and $f\in\HT^{s,\infty}_{FIO}(\Rn)$, Proposition \ref{prop:FIObounds1} yields $e^{it\phi(D)}a(tD)f\in \HT^{s-d(\infty)-m,\infty}(\Rn)\subseteq L^{\infty}(\Rn)\cap C(\Rn)$ for each $t>0$, with
\begin{equation}\label{eq:pinfty}
\sup_{t\in I}\sup_{x\in\Rn}|e^{it\phi(D)}a(tD)f(x)|\lesssim \|f\|_{\HT^{s,\infty}_{FIO}(\Rn)}
\end{equation}
for any compact interval $I\subseteq(0,\infty)$. Moreover, if $f\in\overline{\Sw(\Rn)}\subseteq\HT^{s,\infty}_{FIO}(\Rn)$, then $t\mapsto e^{it\phi(D)}a(tD)f$ is continuous as an $L^{\infty}(\Rn)$-valued map, and \eqref{eq:pointwise1} holds as well. 
Finally, by Theorem \ref{thm:sharpFIO1},  \eqref{eq:pinfty} is essentially sharp under the same assumptions as mentioned below Theorem \ref{thm:maximalFIO1}.

The case $p=\infty$ of the other maximal function bounds in this section 
is straightforward for similar reasons. 
\end{remark}

Next, we consider a variable-coefficient version of Theorem \ref{thm:maximalFIO1}. However, unlike in the latter result, the symbols of the operators under consideration here have compact spatial support.

\begin{theorem}\label{thm:maximalFIO2}
Let $T$ be an operator as in Proposition \ref{prop:FIObounds2}, satisfying \eqref{eq:Gaussmap}. Then, for all $p\in[1,\infty)$ and $s>d(p)+\frac{1}{p}+m$, there exists a $C\geq0$ such that
\begin{equation}\label{eq:maximalFIO2}
\big\|\sup_{t\in \R}|T_{t}f|\big\|_{L^{p}(\Rn)}\leq C\|f\|_{\Hps}
\end{equation}
for all $f\in\Hps$. Hence, for all $t_{0}\in \R$ and for almost all $x\in\Rn$, one has 
\begin{equation}\label{eq:pointwise2}
T_{t}f(x)\to T_{t_{0}}f(x)
\end{equation}
as $t\to t_{0}$.
\end{theorem}

Although \eqref{eq:maximalFIO2}, unlike \eqref{eq:maximalFIO1}, involves the supremum over all $t\in\R$, it should be noted that $T_{t}=0$ for all $t$ outside of a compact interval, by assumption. Similarly, \eqref{eq:pointwise2} holds for all $t_{0}\in\R$ because $T_{t}$ is an FIO of order $m$ for each $t\in\R$.

By Theorems \ref{thm:sharpFIO1} and \ref{thm:sharpFIO2} and Remark \ref{rem:cinematic}, \eqref{eq:maximalFIO2} is essentially sharp for all $p\in[1,2]\cup[\frac{2(n+1)}{n-1},\infty)$ under mild assumptions on $a$.

\begin{proof}
The proof is almost identical to that of Theorem \ref{thm:maximalFIO1}. Firstly, \eqref{eq:pointwise2} follows from \eqref{eq:maximalFIO2}, and it suffices to establish \eqref{eq:maximalFIO2} for $f\in\Sw(\Rn)$, just as in the proof of Theorem \ref{thm:maximalFIO1}. For such $f$ the left-hand side of \eqref{eq:maximalFIO2} is well-defined. 

Now, to prove \eqref{eq:maximalFIO2}, one first considers the low-frequency component $T_{t}r(D)$ of each $T_{t}$,  for $r\in C^{\infty}_{c}(\Rn)$ satisfying $r(\eta)=1$ if $|\eta|\leq 1/2$, and $r(\eta)=0$ if $|\eta|>1$. Write
\[
T_{t}r(D)f=\int_{\Rn}\int_{\Rn}e^{i(x-y)\cdot\eta}e^{i(\Phi(x,t,\eta)-x\cdot\eta)}a(x,t,\eta)r(\eta)\ud \eta f(y)\ud y,
\]
and note that  $(x,\eta)\mapsto e^{i(\Phi(x,t,\eta)-x\cdot\eta)}a(x,t,\eta)r(\eta)$ satisfies the conditions of Lemma \ref{lem:lowfreq}, uniformly in $t\in I$. Hence one can apply Young's inequality and \eqref{eq:Sobolev} to this term, in the same manner as in the proof of Theorem \ref{thm:maximalFIO1}. 

On the other hand, the arguments from the proof of Theorem \ref{thm:maximalFIO1} and the bounds in Proposition \ref{prop:FIObounds2} can be combined to deal with the high frequencies.
\end{proof}

\begin{remark}\label{rem:nocinematic}
If, in Theorem \ref{thm:maximalFIO1}, one does not require that $\rank\,\partial_{\eta\eta}^{2}\phi(\eta)=n-1$ for all $\eta\in V$, then one still has
\begin{equation}\label{eq:nocinematic1}
\big\|\sup_{t\in I}|e^{it\phi(D)}a(tD)f|\big\|_{L^{p}(\Rn)}\lesssim \|f\|_{\Hps}
\end{equation}
for all $p\in[1,\infty)$, $s>s(p)+\frac{1}{p}+m$ and $f\in\Hps$. This follows from the same arguments, except that one cannot rely on \eqref{eq:FIObounds2} to bound the left-hand sides of \eqref{eq:mainproof1} and \eqref{eq:mainproof2}. Instead, for all $\veps>0$ and $f\in\Sw(\Rn)$ such that $\wh{f}(\xi)=0$ for $|\xi|\leq 1/2$, \eqref{eq:FIObounds1} and \eqref{eq:Sobolev} yield
\[
\Big(\int_{I}\|e^{ir\phi(D)}a(rD)f\|_{L^{p}(\Rn)}^{p}\ud r\Big)^{1/p}\lesssim \|f\|_{\HT^{s(p)+m+\veps,p}_{FIO}(\Rn)}
\]
and
\[
\Big(\int_{I}\|\partial_{r}(e^{ir\phi(D)}a(rD)f)\|_{L^{p}(\Rn)}^{p}\ud r\Big)^{1/p}\lesssim \|f\|_{\HT^{s(p)+m+1+\veps,p}_{FIO}(\Rn)},
\]
after which the proof proceeds as before. The appropriate extension of \eqref{eq:pointwise1} holds as well. Note that, for $1\leq p\leq 2$, \eqref{eq:nocinematic1} coincides with \eqref{eq:maximalFIO1}. In this range \eqref{eq:nocinematic1} is essentially sharp under mild assumptions on $a$, as long as $\phi$ is not identically zero on the support of $a$, by Theorem \ref{thm:sharpFIO2}.

A similar remark applies to Theorem \ref{thm:maximalFIO2}, in the case where \eqref{eq:Gaussmap} does not hold. 
\end{remark}

\subsection{Hypersurfaces}\label{subsec:hypersurf}

In this subsection we use Theorems \ref{thm:maximalFIO1} and \ref{thm:maximalFIO2} to prove results about maximal functions associated with hypersurfaces in $\Rn$.

For $\Sigma\subseteq\Rn$ a (smooth) hypersurface, $\psi\in C^{\infty}_{c}(\Rn)$, $f\in C(\Rn)$ and $x\in\Rn$, set
\begin{equation}\label{eq:defA}
\A_{t}f(x):=\int_{\Sigma}f(x-ty)\psi(y)\ud\sigma(y)
\end{equation}
for $t>0$, and
\begin{equation}\label{eq:defM}
\Ma_{\Sigma,\psi}f(x):=\sup_{1\leq t\leq 2}|\A_{t}f(x)|.
\end{equation}
Here and throughout, $\ud\sigma$ is the induced Lebesgue measure on $\Sigma$, defined as in \cite[Section XI.3.1.2]{Stein93} (see also \eqref{eq:defAone}). 
The following result contains Theorem \ref{thm:maximalhypintro} as a special case. By Theorem \ref{thm:sharphyp1}, it is essentially sharp for all $p\in[1,2]\cup [\frac{2(n+1)}{n-1},\infty)$ whenever $\psi$ is not identically zero.

\begin{theorem}\label{thm:maximalhyp1}
Let $\Sigma\subseteq\Rn$ be a hypersurface with non-vanishing Gaussian curvature, and let $\psi\in C^{\infty}_{c}(\Rn)$. Then, for all $p\in[1,\infty)$ and $s>d(p)+\frac{1}{p}-\frac{n-1}{2}$, there exists a $C\geq0$ such that
\begin{equation}\label{eq:maximalhyp1}
\|\Ma_{\Sigma,\psi}f\|_{L^{p}(\Rn)}\leq C\|f\|_{\Hps}
\end{equation}
for all $f\in C(\Rn)\cap \Hps$.
\end{theorem}
In \eqref{eq:defM}, one could replace the supremum over $[1,2]$ by the supremum over an arbitrary compact interval $I\subseteq (0,\infty)$, as follows from the proof. 

\begin{proof}
Fix $f\in C(\Rn)\cap \Hps$. Then the left-hand side of \eqref{eq:maximalhyp1} is well defined. Moreover, by integration $\A_{t}f$ defines a distribution, i.e.~a functional on $C^{\infty}_{c}(\Rn)$, for every $t\in[1,2]$. We claim that $\A_{t}f=b(tD)f$, where $b:=\wh{\psi\ud\sigma}\in C^{\infty}(\Rn)\cap L^{\infty}(\Rn)$ and $b(tD)f$ is initially defined as a distribution as well.

By dilation, it suffices to prove this claim for $t=1$. By definition, one then has
\begin{equation}\label{eq:defAone}
\A_{1}f(x)=\lim_{\veps\to0}\frac{1}{2\veps}\int_{\Rn}f(y)\ind_{S_{\veps}}(x-y)\psi(x-y)\ud y
\end{equation}
for all $x\in\Rn$, where $S_{\veps}:=\{y\in\Rn\mid \inf_{z\in\Sigma}|z-y|<\veps\}$. Hence, if we set $\widetilde{g}(x):=g(-x)$ for $g\in C^{\infty}_{c}(\Rn)$, then 
\begin{align*}
\int_{\Rn}\A_{1}f(x)g(x)\ud x&=\int_{\Rn}\lim_{k\to\infty}\frac{k}{2}\int_{\Rn}f(y)\ind_{S_{1/k}}(x-y)\psi(x-y)g(x)\ud y\ud x\\
&=\lim_{k\to\infty}\frac{k}{2}\int_{\Rn}f(y)\int_{\Rn}\ind_{S_{1/k}}(x-y)\psi(x-y)g(x)\ud x\ud y\\
&=\int_{\Rn}f(y)\lim_{k\to\infty}\frac{k}{2}\int_{\Rn}\ind_{S_{1/k}}(x-y)\psi(x-y)g(x)\ud x\ud y\\
&=\int_{\Rn}f(-y)b(D)\widetilde{g}(y)\ud y,
\end{align*}
and the latter is the adjoint action of the distribution $b(D)f$ on $g$. Note that, in this chain of identities, we used both the dominated convergence theorem and Fubini's theorem. This is in turn allowed because $f$ is locally bounded and, given that $\psi$ and $g$ are compactly supported, the variables $x-y$ and $x$ are restricted to compact sets, meaning that $y$ is as well. This proves the claim.


Now, we will show that $b(tD)f\in L^{p}(\Rn)$ for every $t\in[1,2]$, and that 
\begin{equation}\label{eq:btDestimate}
\big\|\sup_{1\leq t\leq 2}|b(tD)f|\big\|_{L^{p}(\Rn)}\lesssim \|f\|_{\Hps}. 
\end{equation}
It then follows that $\A_{t}f\in L^{p}(\Rn)$ as well, and that \eqref{eq:maximalhyp1} holds, as required.

After decomposing the support of $\psi$ into finitely many pieces and changing coordinates, we may suppose that 
\begin{equation}\label{eq:bdecomp}
b(\xi)=a(\xi)e^{i\phi(\xi)}+e(\xi)
\end{equation}
for all $|\xi|\geq1$ (see e.g.~\cite[Section VIII.5.7]{Stein93} or \cite[Theorem 1.2.1]{Sogge17}). Here $\phi\in C^{\infty}(V)$ is real-valued and positively homogeneous of degree $1$, defined on an open conic subset $V\subseteq\Rn\setminus\{0\}$, and $\rank\,\partial_{\xi\xi}^{2}\phi(\xi)=n-1$ for all $\xi\in V$. Also, $a\in S^{-(n-1)/2}(\Rn)$ satisfies $\supp(a)\subseteq V'\cap\{\xi\in\Rn\mid |\xi|\geq 1/2\}$ for some conic $V'\subseteq V$ which is closed in $\Rn\setminus\{0\}$, and $e\in S^{m}(\Rn)$ for all $m\in\R$. In fact, given that the low-frequency component of $b$ is itself an element of $S^{m}(\Rn)$ for all $m\in\R$, \eqref{eq:bdecomp} holds for all $\xi\in\Rn$.

By Proposition \ref{prop:FIObounds1} and \eqref{eq:bdecomp}, $e^{it\phi(D)}a(tD)f$, $e(tD)f$ and $b(tD)f$ are all elements of $L^{p}(\Rn)$, for every $t\in[1,2]$. Moreover,
\[
\sup_{1\leq t\leq 2}|b(tD)f(x)|\leq \sup_{1\leq t\leq 2}\left(|e^{it\phi(D)}a(tD)f(x)|+|e(tD)f(x)|\right)
\]
for almost all $x\in\Rn$. Next, for each $s>d(p)+\frac{1}{p}-\frac{n-1}{2}$, Theorem \ref{thm:maximalFIO1} yields
\[
\big\|\sup_{1\leq t\leq 2}|e^{it\phi(D)}a(tD)f|\big\|_{L^{p}(\Rn)}\lesssim\|f\|_{\Hps}.
\]
On the other hand, by \eqref{eq:nocinematic1} (with $\phi\equiv 0$ and $m=d(p)-s(p)-\frac{n-1}{2}$), one has
\[
\big\|\sup_{1\leq t\leq 2}|e(tD)f|\big\|_{L^{p}(\Rn)}\lesssim\|f\|_{\Hps}
\]
as well, thereby concluding the proof. 
\end{proof}

\begin{remark}\label{rem:extension}
It is natural to ask whether one can extend \eqref{eq:maximalhyp1} to more general $f\in\Hps$. Here one approach is to recall that, if $f\in C(\Rn)\cap \Hps$, then $\A_{t}f=b(tD)f$ as distributions, where $b=\wh{\psi\ud\sigma}$. Moreover, 
as in the proof of Theorem \ref{thm:maximalhyp1} one can show that, by identifying $\A_{t}f$ and $b(tD)f$, one obtains an extension of \eqref{eq:defM} to general $f\in\Hps$ which makes sense pointwise almost everywhere. Of course, there is a subtlety here regarding the choice of a pointwise representative of $b(tD)f\in L^{p}(\Rn)$ for different values of $t\in[1,2]$, but this can be dealt with in the exact same manner as in the proof of Theorem \ref{thm:maximalFIO1}. More precisely, one can choose these representatives in such a way that $x\mapsto \sup_{1\leq t\leq 2}|b(tD)f(x)|$ is measurable. By approximation, \eqref{eq:maximalhyp1} then extends to all $f\in\Hps$. 


Now, one might wonder whether the extension of \eqref{eq:maximalhyp1} to general $f\in\Hps$ using $b(tD)$ can again be expressed in terms of integrals as in \eqref{eq:defA}. By the reasoning in \cite[Section XI.3.5]{Stein93}, this is indeed the case if $p>\frac{n}{n-1}$ and $f\in L^{p}(\Rn)$, given that then $L^{p}(\Rn)\subseteq\Hps$ continuously for $s\in(d(p)+\frac{1}{p}-\frac{n-1}{2},-s(p)]$. 
On the other hand, if $p>\frac{2}{n-1}$ and $s\in (d(p)+\frac{1}{p}-\frac{n-1}{2},s(p))$, then $\Hps$ is not a subset of $L^{p}(\Rn)$, by the sharpness of the embeddings in \eqref{eq:Sobolev}. In fact, in this case there exist $f\in \Hps$ which are not contained in $L^{p}_{\loc}(\Rn)$; one such example can essentially be found in \cite[Section IX.6.13]{Stein93}. For such $f$ it is less clear how to make sense of the integral in \eqref{eq:defA}. 

Similar comments apply to Theorem \ref{thm:maximalhyp2} and Proposition \ref{prop:geodesicmax} below.
\end{remark}

Next, we consider a version of Theorem \ref{thm:maximalhyp1} where the integral in \eqref{eq:defA} is taken over a hypersurface $\Sigma_{x,t}$ which depends smoothly on $x$ and $t$. More precisely, for $\Psi\in C^{\infty}(\R^{n+1}\times\Rn)$, $x\in\Rn$ and $t\in\R$, set 
\begin{equation}\label{eq:hypvar}
\Sigma_{x,t}:=\{y\in\Rn\mid \Psi_{t}(x,y)=0\}
\end{equation}
where $\Psi_{t}(x,y):=\Psi(x,t,y)$. We will assume that the following two conditions hold:
\begin{enumerate}
\item\label{it:rotational} For all $(x,t)\in\R^{n+1}$ and $y\in\Sigma_{x,t}$, one has $\partial_{x}\Psi_{t}(x,y)\neq 0$, $\partial_{y}\Psi_{t}(x,y)\neq 0$, and
\begin{equation}\label{eq:rotation}
\det\begin{pmatrix}
0&\partial_{y}\Psi_{t}(x,y)\\
\partial_{x}\Psi_{t}(x,y)&\partial_{xy}^{2}\Psi_{t}(x,y)
\end{pmatrix}
\neq0.
\end{equation}
Here we interpret $\partial_{y}\Psi_{t}$ as a row vector and $\partial_{x}\Psi_{t}$ as a column vector.  
\item\label{it:cinematic} For all $(x,t)\in\R^{n+1}$, the conical hypersurface
\begin{equation}\label{eq:cinematic}
\{(\theta\partial_{x}\Psi_{t}(x,y),\theta\partial_{t}\Psi_{t}(x,y))\mid y\in \Sigma_{x,t},\theta\neq0\}\subseteq\R^{n+1}
\end{equation}
has $n-1$ non-vanishing principal curvatures.
\end{enumerate} 

\begin{remark}\label{rem:conditions}
We note that \eqref{it:rotational} is the \emph{rotational curvature} condition from \cite{Phong-Stein86}. It implies that $\Sigma_{x,t}$ is a hypersurface in $\Rn$, and \eqref{eq:rotation} is equivalent to the statement that the Jacobian of $(y,\theta)\mapsto (\theta\partial_{x}\Psi_{t}(x,y),\Psi_{t}(x,y))$ is invertible where $\theta=1$. 
On the other hand, \eqref{it:cinematic} says that the conormal bundle 
\[
\{(x,t,\theta\partial_{x}\Psi_{t}(x,y),\theta\partial_{t}\Psi_{t}(x,y),y,-\theta\partial_{y}\Psi_{t}(x,y))\mid x\in\R^{n},t\in\R, y\in \Sigma_{x,t},\theta\neq 0\}
\]
of $\{(x,t,y)\mid (x,t)\in\R^{n+1},y\in\Sigma_{x,t}\}$ satisfies the cinematic curvature condition from \cite{Sogge91} (see also \cite[Section 8.1]{Sogge17} and Remark \ref{rem:cinematic}).
\end{remark}

Let $\ud\sigma_{x,t}$ be Lebesgue measure restricted to $\Sigma_{x,t}$, defined as before, and let $\rho\in C^{\infty}_{c}(\R^{n+1}\times\Rn)$. For $f\in C(\Rn)$, $x\in\R^{n}$ and $t\in\R$, set
\begin{equation}\label{eq:hypvarop}
T_{t}f(x):=\int_{\Sigma_{x,t}}f(y)\rho(x,t,y)\ud\sigma_{x,t}(y).
\end{equation}
We then obtain a version of Theorem \ref{thm:maximalhyp1} for varying hypersurfaces. By Theorem \ref{thm:sharphyp2} (see also Remark \ref{rem:partialdef}), this result is essentially sharp for all $p\in[1,2]\cup [\frac{2(n+1)}{n-1},\infty)$ whenever $\rho$ is not identically zero.

\begin{theorem}\label{thm:maximalhyp2}
Let $\Psi\in C^{\infty}(\R^{n+1}\times\Rn)$ be such that \eqref{it:rotational} and \eqref{it:cinematic} hold, and let $\rho\in C^{\infty}_{c}(\R^{n+1}\times\Rn)$. For $t\in\R$, let $T_{t}$ be as in \eqref{eq:hypvarop}. Then, for all $p\in[1,\infty)$ and $s>d(p)+\frac{1}{p}-\frac{n-1}{2}$, there exists a $C\geq0$ such that
\begin{equation}\label{eq:maximalhyp2}
\big\|\sup_{t\in \R}|T_{t}f|\big\|_{L^{p}(\Rn)}\leq C\|f\|_{\Hps}
\end{equation}
for all $f\in C(\Rn)\cap \Hps$.
\end{theorem}
\begin{proof}
Again, the left-hand side is well defined for general $f\in C(\Rn)\cap \Hps$. We will reduce to the setting of Theorem \ref{thm:maximalFIO2} by decomposing into finitely many operators and applying changes of coordinates. However, some care has to be taken to deal with the fact that \eqref{eq:maximalhyp2} is not invariant under general changes of coordinates in the $(x,t)$ variable in \eqref{eq:hypvarop}.

Note that
\begin{equation}\label{eq:Lagrangian}
T_{t}f(x)=\frac{1}{2\pi}\int_{\Rn}\int_{\R}e^{i\theta\Psi_{t}(x,y)}\rho_{0}(x,t,y)f(y)\ud\theta\ud y
\end{equation}
for all $f\in C(\Rn)$, $t\in\R$ and $x\in\Rn$, where $\rho_{0}\in C^{\infty}_{c}(\R^{n+1}\times {\R^{n}})$ satisfies $\rho_{0}(x,t,y)=\rho(x,t,y)|\partial_{y}\Psi_{t}(x,y)|$ for $y\in \Sigma_{x,t}$ (see e.g.~\cite[Section XI.3.1.2]{Stein93}). Hence the kernel of $T_{t}$ is a Lagrangian distribution of order $-(n-1)/2$, and $T_{t}$ is a Fourier integral operator of order $-(n-1)/2$ with canonical relation
\[
\Ca_{t}:=\{(x,\theta\partial_{x}\Psi_{t}(x,y),y,-\theta\partial_{y}\Psi_{t}(x,y))\mid x\in\Rn, y\in \Sigma_{x,t}, \theta\neq 0\}.
\]
Condition \eqref{it:rotational} implies that $\Ca_{t}$ is a local canonical graph (see e.g.~\cite[p.~171]{Sogge17}).

Next, set $Tf(x,t):=T_{t}f(x,t)$ for $f\in C(\Rn)$ and $(x,t)\in\R^{n+1}$. Then the canonical relation of $T$ is
\[
\{(x,t,\theta\partial_{x}\Psi_{t}(x,y),\theta\partial_{t}\Psi_{t}(x,y),y,-\theta\partial_{y}\Psi_{t}(x,y))\mid (x,t)\in\R^{n+1}, y\in \Sigma_{x,t}, \theta\neq 0\}.
\] 
As noted in Remark \ref{rem:conditions}, it satisfies the cinematic curvature condition, due to \eqref{it:cinematic}.

Now, by e.g.~\cite[Proposition A.2]{LiRoSoYa24}, there exists a collection $(T_{j})_{j=1}^{l}$ of operators as in Proposition \ref{prop:FIObounds2}, with $m=-(n-1)/2$, changes of coordinates $(S_{j})_{j=1}^{l}$ on $\Rn$, and an operator $R$ with a kernel which is a Schwartz function, such that $T=\sum_{j=1}^{l}T_{j}S_{j}+R$ (that no changes of coordinates on $\R^{n+1}$ are required is noted in \cite[Remark A.3]{LiRoSoYa24}). The smoothing operator $R$ satisfies the required maximal function estimate, since it maps $\Hps$ into $\Sw(\R^{n+1})$ continuously. Moreover, $\Hps$ is invariant under the changes of coordinates $S_{j}$, and each $T_{j}$ satisfies the required maximal function estimate, by Theorem \ref{thm:maximalFIO2}. 
This concludes the proof.
\end{proof}

\begin{remark}\label{rem:specificsurface}
Clearly, $\Psi$ does not need to be defined on all of $\R^{n+1}\times\Rn$. It suffices for conditions \eqref{it:rotational} and \eqref{it:cinematic} to hold for all $(x,t,y)$ in a neighborhood of $\supp(\rho)$.

Moreover, if $\Psi$ satisfies \eqref{it:rotational} but not \eqref{it:cinematic}, then one still has
\[
\big\|\sup_{t\in \R}|T_{t}f|\big\|_{L^{p}(\Rn)}\lesssim\|f\|_{\Hps}
\]
for all $p\in[1,\infty)$, $s>s(p)+\frac{1}{p}-\frac{n-1}{2}$ and $f\in C(\Rn)\cap \Hps$. This follows as before, if one relies on Remark \ref{rem:nocinematic} instead of Theorem \ref{thm:maximalFIO2}. 

Similarly, suppose that $\Sigma\subseteq \Rn$ is a hypersurface of the form $\Sigma=\{y\in\Rn\mid \phi(y)=1\}$, for some real-valued $\phi\in C^{\infty}(V)$, positively homogeneous of degree $1$ and defined on an open subset $V\subseteq\Rn\setminus\{0\}$. Then, even if one does not assume that $\rank\,\partial_{\xi\xi}^{2}\phi(\xi)=n-1$ for all $\xi\in V$, for each $\psi\in C^{\infty}_{c}(\Rn)$ one still has
\[
\|M_{\Sigma,\psi}f\|_{L^{p}(\Rn)}\lesssim\|f\|_{\Hps}
\]
for all $p\in[1,\infty)$, $s>s(p)+\frac{1}{p}-\frac{n-1}{2}$ and $f\in C(\Rn)\cap\Hps$. This follows as in the proof of Theorem \ref{thm:maximalhyp1}, but relying on \eqref{eq:nocinematic1} instead of Theorem \ref{thm:maximalFIO1}.
\end{remark}

\subsection{Spherical means of complex order}\label{subsec:complex}

In this subsection we derive from Theorem \ref{thm:maximalFIO1} maximal bounds for the complex spherical means from \cite{Stein76}.

For $\alpha,t>0$, $f\in C(\Rn)$ and $x\in\Rn$, set \begin{equation}\label{eq:defMalpha}
\Ma^{\alpha}_{t}f(x):=\frac{1}{\Gamma(\alpha)}\int_{B_{1}(0)}(1-|y|^{2})^{\alpha-1}f(x-ty)\ud y.
\end{equation}
Here $B_{1}(0)$ is the unit ball in $\Rn$. Note that $\Ma^{\alpha}_{t}f=K_{\alpha,t}\ast f$, where $K_{\alpha,t}(x):=t^{-n}K_{\alpha}(t^{-1}x)$ and $K_{\alpha}(x):=\Gamma(\alpha)^{-1}\max((1-|x|^{2})^{\alpha-1},0)$ for $x\in\Rn$. One has 
\begin{equation}\label{eq:defmalpha}
m_{\alpha}(\xi):=\pi^{-\alpha+1}|\xi|^{-n/2-\alpha+1}J_{n/2+\alpha-1}(2\pi|\xi|)=\wh{K_{\alpha}}(\xi)
\end{equation}
for all $\xi\neq 0$, where $J_{\beta}$ is the Bessel function of the first kind, of order $\beta\in\C$ (see \cite[Theorem IV.4.15]{Stein-Weiss71}). For a general $\alpha\in\C$ one can then also define $m_{\alpha}$ as in \eqref{eq:defmalpha}. In turn, this then leads to the definition of the complex spherical mean $\Ma^{\alpha}_{t}$, by setting 
\begin{equation}\label{eq:complexmean}
\Ma^{\alpha}_{t}f:=m_{\alpha}(tD)f
\end{equation}
for $f\in\Sw(\Rn)$. 

We now obtain a maximal function bound for the complex spherical means. This result is essentially sharp for all $p\in[1,2]\cup [\frac{2(n+1)}{n-1},\infty)$, cf.~Proposition \ref{prop:sharpcomplex}.

\begin{proposition}\label{prop:complexsphere}
Let $\alpha\in\C$. Then, for all $p\in[1,\infty)$ and $s>d(p)+\frac{1}{p}-\frac{n-1}{2}-\Real(\alpha)$, and for each compact interval $I\subseteq(0,\infty)$, there exists a $C\geq0$ such that
\begin{equation}\label{eq:complexsphere}
\big\|\sup_{t\in I}|\Ma^{\alpha}_{t}f|\big\|_{L^{p}(\Rn)}\leq C\|f\|_{\Hps}
\end{equation}
for all $f\in\Hps$.
\end{proposition}
\begin{proof}
Note that, as follows from asymptotics of Bessel functions (see \cite[Chapters 3 and 7]{Watson95}), there exist $a^{0}_{\alpha}\in \cap_{m\in {\mathbb R}}S^m(\Rn)$ and $a^{+}_{\alpha},a^{-}_{\alpha}\in S^{-\frac{n-1}{2}-\Real(\alpha)}(\Rn)$ such that
\begin{equation}\label{eq:malpha}
m_{\alpha}(\xi)=a_{\alpha}^{0}(\xi)+e^{i|\xi|}a^{+}_{\alpha}(\xi)+e^{-i|\xi|}a^{-}_{\alpha}(\xi)
\end{equation}
for all $\xi\in\Rn$. In particular, $m_{\alpha}(tD):\Sw(\Rn)\to\Sw(\Rn)$ continuously for each $t>0$, and therefore $m_{\alpha}(tD):\Sw'(\Rn)\to\Sw'(\Rn)$ continuously as well. Hence $\Ma^{\alpha}_{t}f$ is a priori well defined, as a tempered distribution, for every $f\in\Hps$.

Now, as in the proof of Theorem \ref{thm:maximalFIO1}, if $f\in\Sw(\Rn)$ then
\[
\big\|\sup_{t\in I}|a^{0}_{\alpha}(tD)f|\big\|_{L^{p}(\Rn)}\lesssim \|f\|_{\Hps}
\]
is a well-defined inequality which is valid. Theorem \ref{thm:maximalFIO1} yields the required bounds for the other terms in \eqref{eq:malpha}. Finally, as in the proof of Theorem \ref{thm:maximalFIO1}, by approximation \eqref{eq:complexsphere} extends in a well-defined manner to general $f\in \Hps$.
\end{proof}

\begin{remark}\label{rem:extensioncomplex}
The expression for $\Ma^{\alpha}_{t}f$ in \eqref{eq:defMalpha}, for $\alpha>0$ and $f\in C(\Rn)$, extends to $\alpha\in\C$ with $\Real(\alpha)>0$. Hence, if $f\in C(\Rn)\cap \Hps$, then Proposition \ref{prop:complexsphere} yields maximal function bounds for these integrals. As in Remark \ref{rem:extension}, these estimates can in turn be extended to $f\in L^{\infty}_{\loc}(\Rn)\cap \Hps$.
\end{remark}

\begin{remark}\label{rem:meanexample}
For $\alpha=0$, one recovers the spherical maximal function from \eqref{eq:defMsphere}. On the other hand, for $\alpha=1$, Proposition \ref{prop:complexsphere} and the previous remark yield
\begin{equation}\label{eq:HL}
\Big(\int_{\Rn}\sup_{t\in I}\Big|\fint_{B_{t}(x)}f(y)\ud y\Big|^{p}\ud x\Big)^{1/p}\lesssim \|f\|_{\Hps}
\end{equation}
for all $1\leq p<\infty$, $s>d(p)+\frac{1}{p}-\frac{n+1}{2}$ and $f\in C(\Rn)\cap \Hps$. If, additionally, $f$ is non-negative, then the left-hand side of \eqref{eq:HL} is a version of the (centered) Hardy--Littlewood maximal operator where the radii of balls are restricted to a compact interval in $(0,\infty)$.
\end{remark}

\subsection{Maximal functions on manifolds}\label{subsec:manifold}

In this subsection we derive versions of Theorems \ref{thm:maximalhypintro} and \ref{thm:maximalFIOintro} on compact manifolds. 

Throughout, fix an $n$-dimensional compact Riemannian manifold $(M,g)$ without boundary, and let $\Delta_{g}$ be the associated (non-positive) Laplace--Beltrami operator. Then $\Delta_{g}$ acts on half densities on $M$, as do the operators $e^{it\sqrt{-\Delta_{g}}}$, defined for all $t\in\R$ through the spectral theorem on $L^{2}(M,\Omega_{1/2})$. 

We first consider an analogue of Theorem \ref{thm:maximalFIOintro} on $M$. Due to the fact that $L^{p}(M)$ is a proper subspace of $\HpsM$, this result improves upon \cite[Theorem 46]{BeHiSo21}. 

\begin{proposition}\label{prop:maximalFIO3}
Let $p\in[1,\infty)$ and $s>d(p)+\frac{1}{p}$. Then, for each bounded interval $I\subseteq\R$, there exists a $C\geq0$ such that
\begin{equation}\label{eq:maximalFIO3}
\big\|\sup_{t\in I}|e^{it\sqrt{-\Delta_{g}}}f|\big\|_{L^{p}(M)}\leq C\|f\|_{\HpsM}
\end{equation}
for all $f\in\HpsM$. Hence, for all $t_{0}\in \R$ and for almost all $x\in M$, one has 
\begin{equation}\label{eq:pointwise3}
e^{it\sqrt{-\Delta_{g}}}f(x)\to e^{it_{0}\sqrt{-\Delta_{g}}}f(x)
\end{equation}
as $t\to t_{0}$.
\end{proposition}
\begin{proof}
Firstly, \eqref{eq:pointwise3} follows from \eqref{eq:maximalFIO3} as before, by approximating a given $f\in\Hps$ by elements of $\Da(M,\Omega_{1/2})$, for which \eqref{eq:pointwise3} holds for every $x\in M$.

Now, for $f\in\Da(M,\Omega_{1/2})$, $t\in\R$ and $x\in M$, set $Tf(x,t):=e^{it\sqrt{-\Delta_{g}}}f(x)$. By e.g.~\cite[Theorem 29.1.1 and Proposition 29.1.9]{Hormander09}, $T$ is an FIO of order $-1/4$ from $M$ to $M\times \R$, with canonical relation
\begin{equation}\label{eq:canwave}
\Ca=\{(x,t,\xi,\tau,y,\eta)\mid (x,\xi)=\chi_{t}(y,\eta),\tau=|\xi|_{g}\}.
\end{equation}
Here $(\chi_{t})_{t\in \R}$ is the geodesic flow on $T^{*}M$. In particular, each $e^{it\sqrt{-\Delta_{g}}}$ is a Fourier integral operator of order $0$, the canonical relation of which is the graph of $\chi_{t}$. This in turn means that $e^{it\sqrt{-\Delta_{g}}}$ acts continuously on $\Da(M,\Omega_{1/2})$ and $\Da'(M,\Omega_{1/2})$, and $e^{it\sqrt{-\Delta_{g}}}f\in\Da'(M,\Omega_{1/2})$ is a priori well defined for each $f\in\HpsM$. Moreover, if $f\in\Da(M,\Omega_{1/2})$ then the left-hand side of \eqref{eq:maximalFIO3} is well defined. 

Finally, $\Ca$ satisfies the cinematic curvature condition (see e.g.~\cite[Section 8.1]{Sogge17}). Thus, by working in local coordinates, one can argue as in the proof of Theorem \ref{thm:maximalhyp2} to derive the required statement from Theorem \ref{thm:maximalFIO2}. 
\end{proof}

Next, let $\inj(M)>0$ be the injectivity radius of $M$. For $x\in M$ and $t\in(0,\inj(M))$, let $B_{t}(x)\subseteq M$ be the (open) geodesic ball around $x$ of radius $t$, and let $S_{t}(x)$ be its boundary, the geodesic sphere of radius $t$ around $x$. Set
\begin{equation}\label{eq:FIOgeodesic}
\A_{t,g}f(x):=\int_{S_{t}(x)}f(y)\sqrt{dV_{g}}
\end{equation}
for $t\in(0,\inj(M))$, $f\in C(M,\Omega_{1/2})$ and $x\in M$, where we identify $\sqrt{dV_{g}}$ with its restriction to $S_{t}(x)$, defined in an analogous manner as in Section \ref{subsec:hypersurf}. 

\begin{proposition}\label{prop:geodesicmax}
Let $p\in[1,\infty)$ and $s>d(p)+\frac{1}{p}-\frac{n-1}{2}$. Then, for each compact interval $I\subseteq (0,\inj(M))$, there exists a $C\geq 0$ such that
\[
\big\|\sup_{t\in I}|\A_{t,g}f|\big\|_{L^{p}(M)}\leq C\|f\|_{\HpsM} 
\]
for all $f\in C(M,\Omega_{1/2})\cap \HpsM$.
\end{proposition}
\begin{proof}
If one sets $Tf(x,t):=\A_{t,g}f(x)$ for $f\in\Da(M,\Omega_{1/2})$ and $(x,t)\in M\times I$, then the canonical relation of $T$ is the union of the relation in \eqref{eq:canwave} and the corresponding one with the $\tau$ coordinate replaced by $-\tau$. Since each $\A_{t,g}$ is an FIO of order $-(n-1)/2$, one can now argue as in the proof of Proposition \ref{prop:maximalFIO3}.

Alternatively, one can reason as follows. It suffices to work in a sufficiently small geodesic normal coordinate neighborhood of a given point $x_{0}\in M$. There each $\A_{t,g}$ is as in \eqref{eq:hypvarop}, with $\Sigma_{x_{0},t}=x_0+tS^{n-1}$ and $\Psi_{t}(x_{0},y)=|x_{0}-y|-t$. Then $\Psi_{t}(x,y)$ is a small perturbation of $\tilde{\Psi}_{t}(x,y):=|x-y|-t$ for $x$ near $x_{0}$. Since $\tilde{\Psi}_{t}$ satisfies conditions \eqref{it:rotational} and \eqref{it:cinematic}, and because these conditions are stable under small perturbations, the required statement follows from Theorem \ref{thm:maximalhyp2}.
\end{proof}

\section{Sharpness}\label{sec:sharpness}

In this section we show that the results from the previous section are essentially sharp for suitable values of $p\in[1,\infty]$. We first consider the maximal function bounds for Fourier integral operators. We then prove sharpness of the maximal function estimates for hypersurfaces, and for complex spherical averages, by reducing to the corresponding statements for Fourier integral operators. Finally, we obtain sharpness of the maximal function bounds on manifolds by working in local coordinates.

\subsection{Fourier integral operators}\label{subsec:sharpFIO}

Throughout this subsection, we consider operators as in \eqref{eq:Tstandard}. More precisely, let $V\subseteq \R^{n+1}\times(\R^{n}\setminus\{0\})$ be an open conic subset, and let $\Phi\in C^{\infty}(V)$ be real-valued and positively homogenous of degree $1$ in the fiber variable. Let $m\in\R$ and $a\in S^{m}(\R^{n+1}\times\Rn)$ be such that $(z,\eta)\in V$ whenever $(z,\eta)\in\supp(a)$ and $\eta\neq0$. Suppose that $\rank\,\partial_{x\eta}^{2}\Phi(z_{0},\eta_{0})=n$ for all $(z_{0},\eta_{0})\in\supp(a)$ with $\eta_{0}\neq 0$, where we write $z=(x,t)\in\Rn\times\R$. Set
\begin{equation}\label{eq:Tstandard2}
Tf(z):=\int_{\Rn}e^{i\Phi(z,\eta)}a(z,\eta)\wh{f}(\eta)\ud\eta
\end{equation}
for $f\in\Sw(\Rn)$ and $z\in\R^{n+1}$, and $T_{t}f(x):=Tf(x,t)$. Then, cf.~\cite{Hormander09}, $T$ is a Fourier integral operator of order $m-1/4$, and each $T_{t}$ is a Fourier integral operator of order $m$ associated with a local canonical graph.

To prove sharpness of the results in Section \ref{subsec:FIOmaximal}, we will rely on the following lemma, an improvement of \cite[Lemma 7.1]{LiRoSoYa24}.

\begin{lemma}\label{lem:flow}
Suppose that $m=0$ and that there exists a non-empty conic $V'\subseteq\Rn\setminus\{0\}$ such that $a(0,\eta)=1$ and $\Phi(0,\eta)=0$
for all $\eta\in V'$. 
Then there exists a $C\geq0$ such that
\[
|Th(z)-(2\pi)^{n}h(\partial_{\eta}\Phi(z,\nu))|\leq C\|\wh{h}\|_{L^{1}(\Rn)}|z|(1+\gamma(h,\nu))
\]
for all $h\in \Sw(\Rn)$ with $\supp(\wh{h})\subseteq V'$ 
compact, and all $\nu\in V'\cap S^{n-1}$ and $z\in\R^{n+1}$. Here
\[
\gamma(h,\nu):=\sup\{|\hat{\eta}-\nu|^{2}|\eta|\mid \eta\in\supp(\wh{h})\}.
\]
\end{lemma}
\begin{proof}
The proof follows that of \cite[Lemma 7.1]{LiRoSoYa24}, but ends with an additional step. 

Fix $\nu\in S^{n-1}\cap V'$ and $z\in\R^{n+1}$, and write
\begin{align*}
&Th(z)-(2\pi)^{n}h(\partial_{\eta}\Phi(z,\nu))=\int_{\Rn}\big(e^{i\Phi(z,\eta)}a(z,\eta)-e^{i\partial_{\eta}\Phi(z,\nu)\cdot\eta}\big)\wh{h}(\eta)\ud\eta\\
&=\int_{\Rn}e^{i\Phi(z,\eta)}(a(z,\eta)-1)\wh{h}(\eta)\ud\eta+\int_{\Rn}\big(e^{i\Phi(z,\eta)}-e^{i\partial_{\eta}\Phi(z,\nu)\cdot\eta}\big)\wh{h}(\eta)\ud\eta.
\end{align*}
We will bound each of these terms separately. For the first term, simply note that
\[
\Big|\int_{\Rn}e^{i\Phi(z,\eta)}(a(z,\eta)-1)\wh{h}(\eta)\ud\eta\Big|\lesssim \|\wh{h}\|_{L^{1}(\Rn)}|z|.
\]
Here we used Taylor approximation and the assumptions that $m=0$ and that $a(0,\eta)=1$ for all $\eta\in \supp(\wh{h})$. 

On the other hand, for the second term we use the homogeneity of $\Phi$ to write:
\begin{align*}
&\Big|\int_{\Rn}\big(e^{i\Phi(z,\eta)}-e^{i\partial_{\eta}\Phi(z,\nu)\cdot\eta}\big)\wh{h}(\eta)\ud\eta\Big|\leq \int_{\Rn}\big|e^{i(\Phi(z,\eta)-\partial_{\eta}\Phi(z,\nu)\cdot\eta)}-1\big|\,|\wh{h}(\eta)|\ud\eta\\
&=\int_{\Rn}\big|e^{i(\partial_{\eta}\Phi(z,\hat\eta)-\partial_{\eta}\Phi(z,\nu))\cdot\eta}-1\big|\,|\wh{h}(\eta)|\ud\eta\leq \sqrt{2}\|\wh{h}\|_{L^{1}(\Rn)}\delta(h,\nu,z),
\end{align*}
where
\[
\delta(h,\nu,z):=\sup\{|\eta\cdot(\partial_{\eta}\Phi(z,\nu)-\partial_{\eta}\Phi(z,\hat\eta))|\mid \eta\in\supp(\wh{h})\}.
\]
Here we also used that $|e^{ix}-1|=|\cos(x)-1+i\sin(x)|\leq \sqrt{2}|x|$ for all $x\in\R$. 

It thus remains to show that $\delta(h,\nu,z)\lesssim |z|\gamma(h,\nu)$. To this end, for $\eta\in\supp(\wh{h})$, we apply another Taylor expansion and write $\partial_{\eta}\Phi(z,\nu)-\partial_{\eta}\Phi(z,\hat{\eta})$ as
\[
\partial_{\eta\eta}^{2}\Phi(z,\hat{\eta})(\nu-\hat{\eta})+\sum_{\alpha\in\Z_{+}^{n},|\alpha|=2}\int_{0}^{1}(1-r)\partial_{\eta}^{\alpha}\partial_{\eta}\Phi(z,(1-r)\nu+r\hat{\eta})\ud r\cdot (\nu-\hat{\eta})^{\alpha}.
\]
Moreover, $\partial_{\eta_{i}}\Phi(z,\eta)$ is positively homogeneous of order zero in $\eta$ for every $1\leq i\leq n$, because $\Phi(z,\eta)$ is positively homogeneous of order one in $\eta$. In particular, the radial derivative $(\eta\cdot\partial_{\eta})\partial_{\eta_{i}}\Phi(z,\hat{\eta})$ vanishes for every $1\leq i\leq n$. Hence 
\[
\eta\cdot\big(\partial_{\eta\eta}^{2}\Phi(z,\hat{\eta})(\nu-\hat{\eta})\big)=\big(\eta^{t}\partial_{\eta\eta}^{2}\Phi(z,\hat{\eta})\big)\cdot(\nu-\hat{\eta})=\big((\eta\cdot\partial_{\eta})\partial_{\eta_{i}}\Phi(z,\hat{\eta})\big)_{i=1}^{n}\cdot (\nu-\hat{\eta})=0.
\]
We thus obtain that
\[
|\eta\cdot(\partial_{\eta}\Phi(z,\nu)-\partial_{\eta}\Phi(z,\hat\eta))|\leq |\eta| \sum_{\alpha\in\Z_{+}^{n},|\alpha|=2}\int_{0}^{1}|\partial_{\eta}^{\alpha}\partial_{\eta}\Phi(z,(1-r)\nu+r\hat{\eta})|\ud r |\nu-\hat{\eta}|^{2}.
\]
Finally, note that $\partial_{\eta}^{\alpha}\partial_{\eta}^{2}\Phi(0,(1-r)\nu+r\hat{\eta})=0$ for all $\alpha\in\Z_{+}^{n}$ and $0\leq r\leq 1$, since $\Phi(0,\eta')= 0$ for all $\eta'\in V'$. 
Hence for each $r$ we can use another Taylor approximation, this time with respect to $z$, to obtain
\[
|\eta\cdot(\partial_{\eta}\Phi(z,\nu)-\partial_{\eta}\Phi(z,\hat{\eta}))|\lesssim |z|\,|\hat{\eta}-\nu|^{2}|\eta|.\qedhere
\]
\end{proof}

Now, one clearly has 
\[
\|T_{t_{0}}f\|_{L^{p}(\Rn)}\leq \big\|\sup_{t\in\R}|T_{t}f|\big\|_{L^{p}(\Rn)}
\]
for any $t_{0}\in \R$ and $f\in\Sw(\Rn)$, and $d(p)+\frac{1}{p}=s(p)=\frac{n-1}{2}(\frac{1}{2}-\frac{1}{p})$ for $p\geq \frac{2(n+1)}{n-1}$. Hence the following result shows that, under mild conditions on the symbol of the operator, the maximal function bounds in Theorems \ref{thm:maximalFIO1} and \ref{thm:maximalFIO2}, and in Corollary \ref{cor:pointwiseFIO1} and Remark \ref{rem:pinfty}, are essentially sharp for such $p$. 

\begin{theorem}\label{thm:sharpFIO1}
Let $p\in[1,\infty]$ and $s\in\R$. Suppose that there exist a $z_{0}=(x_{0},t_{0})\in\R^{n+1}$, a non-empty open conic set $V'\subseteq \Rn\setminus\{0\}$ and $c',C'>0$ such that $|a(z_{0},\eta)|\geq c'|\eta|^{m}$ for all $\eta\in V'$ with $|\eta|\geq C'$. If there exists a $C\geq0$ such that
\begin{equation}\label{eq:sharpFIO1}
\|T_{t_{0}}f\|_{L^{p}(\Rn)}\leq C\|f\|_{\HT^{s,p}_{FIO}(\Rn)}
\end{equation}
for all $f\in\Sw(\Rn)$, then $s\geq \frac{n-1}{2}(\frac{1}{2}-\frac{1}{p})+m$. 
\end{theorem}
\begin{proof}
The proof is analogous to that of \cite[Theorem 7.3]{LiRoSoYa24}. However, given that the setting under consideration presently is slightly different, and because we will use some of the ideas later on, we provide most of the details.

We first make several reduction steps. By translating and precomposing with $\lb D\rb^{-m}$, we may suppose that $z_{0}=(x_{0},t_{0})=0$ and $m=0$. Then, by adding a harmless smoothing term which modifies the low frequencies of $a$, we may suppose that $|a(0,\eta)|\geq c'$ for all $\eta\in V'$. 

Next, by shrinking $V'$, 
we may suppose that the closure of $V'$ in $\R^{n}\setminus\{0\}$ is a strict subset of a closed cone $V_{0}\subseteq\Rn\setminus\{0\}$ such that $|a(0,\eta)|\geq c'$ for all $\eta\in V_{0}$. 
Let $\rho\in C^{\infty}(\R^{n})$ be such that $\supp(\rho)\subseteq V_{0}$ 
and $\rho(\eta)=1$ for all $\eta\in V'$, 
and set
\[
\theta(\eta):=e^{-i\Phi(0,\eta)}\frac{\rho(\eta)}{a(0,\eta)}
\]
for $\eta\in\Rn$. By Proposition \ref{prop:FIObounds1}, $\theta(D)$ is bounded on $\Hps$, and \eqref{eq:sharpFIO1} implies that $\|T_{0}\theta(D)f\|_{L^{p}(\Rn)}\lesssim\|f\|_{\HT^{s,p}_{FIO}(\Rn)}$ for all $f\in\Sw(\Rn)$. Moreover, $T\theta(D)$ is an operator as in \eqref{eq:Tstandard2}, with phase function $\wt{\Phi}$ and symbol $\wt{a}$ satisfying $\wt{\Phi}(0,\eta)=0$ and $\wt{a}(0,\eta)=1$ for $\eta\in V'$. 
By replacing $T$ by $T\theta(D)$, we may thus suppose in the remainder that $\Phi(0,\eta)=0$ and $a(0,\eta)=1$ for all $\eta\in V'$. 

After these reduction steps, we get to the heart of the proof. For each $k\in\Z_{+}$, let $\Theta_{k}\subseteq S^{n-1}$ be a maximal collection of unit vectors such that $|\nu-\nu'|\geq 2^{-k/2}$ for all $\nu,\nu'\in \Theta_{k}$ with $\nu\neq \nu'$. Let $V''\subseteq V'$ be a non-empty cone which is closed in $V'$, and fix a $k\in\Z_{+}$ such that $V_{k}:=\Theta_{k}\cap V''$ has approximately $2^{k(n-1)/2}$ elements. Note that this is true for all large $k$, since $V''$ has a nonzero aperture. 
Next, let $(\chi_{\nu})_{\nu\in\Theta_{k}}\subseteq C^{\infty}(\Rn\setminus\{0\})$ be an associated partition of unity. In particular, each $\chi_{\nu}$ is positively homogeneous of degree $0$ and satisfies $0\leq \chi_{\nu}\leq 1$,
\[
\supp(\chi_{\nu})\subseteq\{\eta\in\Rn\setminus\{0\}\mid |\hat{\eta}-\nu|\leq 2^{-k/2+1}\},
\]
and $\sum_{\nu\in \Theta_{k}}\chi_{\nu}(\eta)=1$ for all $\eta\neq0$. We may also suppose that $\chi_{\nu}(\eta)=1$ whenever $|\hat{\eta}-\nu|\leq 2^{-k/2-1}$. 
Finally, fix a $\psi\in \Sw(\Rn)$ such that $\wh{\psi}$ is non-negative and not identically zero, 
and such that $\wh{\psi}(\eta)=0$ if $|\eta|>c$, for some small constant $c>0$.

Now, for $\nu\in V_{k}$ and $x\in\Rn$, set
\begin{equation}\label{eq:fnudef}
f_{\nu}(x):=e^{i2^{k}\nu\cdot x}\psi(2^{k}(\nu\cdot x)\nu+2^{k/2}\Pi_{\nu}^{\perp}x),
\end{equation}
where $\Pi_{\nu}^{\perp}$ is the orthogonal projection onto the complement of the span of $\nu$. Since $V''$ is closed in $V'$, by choosing $c$ sufficiently small, we may suppose that
\[
\supp(\wh{f_{\nu}})\subseteq V'\cap \{\eta\in\Rn\mid 2^{k-1}\leq |\eta|\leq 2^{k+1},|\hat{\eta}-\nu|\leq 2^{-k/2-1}\}
\]
for each $\nu$. In particular, $\chi_{\nu}(D)f_{\nu}=f_{\nu}$. Note also that $\|\wh{f_{\nu}}\|_{L^{1}(\Rn)}\eqsim 1$ and $\|f_{\nu}\|_{L^{p}(\Rn)}\eqsim 2^{-k\frac{n+1}{2p}}$, for implicit constants independent of $\nu$ and $k$. Set $f:=\sum_{\nu\in V_{k}}f_{\nu}\in\Sw(\Rn)$. Then $\supp(\wh{f}\,)\subseteq V'\cap \{\xi\in\Rn\mid 2^{k-1}\leq |\eta|\leq 2^{k+1}\}$ and 
\begin{equation}\label{eq:fnorm}
\|f\|_{\Hps}\eqsim 2^{k(s+\frac{n-1}{2}(\frac{1}{2}-\frac{1}{p}))}\Big(\sum_{\nu\in\Theta_{k}}\|f_{\nu}\|_{L^{p}(\Rn)}^{p}\Big)^{1/p}\eqsim 2^{k(s+\frac{n-1}{2}(\frac{1}{2}-\frac{1}{p})-\frac{1}{p})},
\end{equation}
with the obvious modification for $p=\infty$, as follows from \cite[Proposition 4.1]{Rozendaal22b}. 

Next, note that $\Phi$ and all of its derivatives are bounded on compact subsets of $\R^{n+1}\times(\Rn\setminus\{0\})$. Hence $z\mapsto \partial_{\eta}\Phi(z,\nu)$ is a locally Lipschitz flow for each $\nu\in V_{k}$, with Lipschitz constants independent of $\nu$. Moreover, by assumption, $\partial_{\eta}\Phi(0,\nu)=0$ and $f_{\nu}(0)=\|\wh{\psi}\|_{L^{1}(\Rn)}>0$ for each $\nu$. It thus follows that $\Real(f_{\nu}(\partial_{\eta}\Phi(x,0,\nu)))\gtrsim 1$ whenever $|x|\lesssim 2^{-k}$ for some implicit constant independent of $k$ and $\nu$. On the other hand, $|\hat{\eta}-\nu|^{2}|\eta|\lesssim 1$ for all $\eta\in\supp(\wh{f_{\nu}}\,)$. Hence Lemma \ref{lem:flow} yields
\[
|T_{0}f_{\nu}(x)-(2\pi)^{n}f_{\nu}(\partial_{\eta}\Phi(x,0,\nu))|\lesssim |x|
\]
for all $x\in\Rn$. We thus see that
\[
\Real T_{0}f_{\nu}(x)\geq (2\pi)^{n}\Real f_{\nu}(\partial_{\eta}\Phi(x,0,\nu))-|T_{0}f_{\nu}(x)-(2\pi)^{n}f_{\nu}(\partial_{\eta}\Phi(x,0,\nu))|\gtrsim 1
\]
if $|x|\lesssim 2^{-k}$. Finally, we can combine \eqref{eq:sharpFIO1} and \eqref{eq:fnorm}:
\begin{align*}
2^{-k\frac{n}{p}}2^{k\frac{n-1}{2}}&\lesssim \Big(\int_{\R^{n}}\Big|\sum_{\nu\in V_{k}}T_{0}f_{\nu}(x)\Big|^{p}\ud x\Big)^{1/p}=\|T_{0}f\|_{L^{p}(\R^{n})}\lesssim \|f\|_{\HT^{s,p}_{FIO}(\Rn)}\\
&\eqsim 2^{k(s+\frac{n-1}{2}(\frac{1}{2}-\frac{1}{p})-\frac{1}{p})},
\end{align*}
with the obvious modification for $p=\infty$. This implies the required statement.
\end{proof}

\begin{remark}\label{rem:singleFIO}
Note that Theorem \ref{thm:maximalFIO1} is in fact a statement about a single Fourier integral operator acting on function spaces over $\Rn$, and the implicit time parameter plays no role in the proof. Combined with the Sobolev embeddings in \eqref{eq:Sobolev}, the theorem shows that the bounds in Proposition \ref{prop:FIObounds1} \eqref{it:FIObounds11} and Proposition \ref{prop:FIObounds2} \eqref{it:FIObounds21} are sharp under very mild assumptions, for $p\in[2,\infty)$. 
It was already shown in \cite{Rozendaal22b,LiRoSoYa24} that Proposition \ref{prop:FIObounds1} \eqref{it:FIObounds12} and Proposition \ref{prop:FIObounds2} \eqref{it:FIObounds22} are sharp under the same assumptions. 
\end{remark}

\begin{remark}\label{rem:SSS}
The conclusion of Theorem \ref{thm:sharpFIO1} essentially follows from \cite{SeSoSt91} (see also \cite[pp.~185-186]{Sogge17}) if the projection of the canonical relation
\[
\{(x,\partial_{x}\Phi(x,t_{0},\eta)),(\partial_{\eta}\Phi(x,t_{0},\eta),\eta)\mid (x,\eta)\in V\}\subseteq T^{*}(\Rn)\times T^{*}(\Rn)
\]
of $T_{t_{0}}$ onto the base space $\Rn\times \Rn$ has full rank at $x_{0}$, and in that case we may replace \eqref{eq:sharpFIO1} by the weaker assumption that 
\begin{equation}\label{eq:sharpclassical}
T_{t_{0}}:\HT^{s+s(p),p}(\Rn)\to L^{p}(\Rn)
\end{equation}
is bounded. This assumption on the rank of the projection is satisfied in the setting of Theorem \ref{thm:maximalFIOintro} and \eqref{eq:maximalFIOintro3}, and, combined with some of the arguments later on in this section, also in \eqref{eq:localimprove} and in the results in Sections \ref{subsec:hypersurf}, \ref{subsec:complex} and \ref{subsec:manifold}. On the other hand, \eqref{eq:sharpclassical} definitely does not imply that $s\geq s(p)$ for general $\Phi$ as in Theorem \ref{thm:sharpFIO1}, since the latter result also applies when, for example, $T_{t_{0}}$ is the identity operator.
\end{remark}

Next, we consider sharpness of the results in Section \ref{subsec:FIOmaximal} for $p\in[1,2]$. Here one cannot expect to rely on an approach as general as in Theorem \ref{thm:sharpFIO1}. Indeed, let $T_{t_{0}}$ be a Fourier integral operator of order $m$ on $\Rn$, associated with a local canonical graph and with a principal symbol which is not identically zero. Then one may set $Tf(x,t):=\chi(t)T_{t_{0}}f(x)$ for $(x,t)\in\R^{n+1}$, where $\chi\in C^{\infty}_{c}(\R)$ is not identically zero. Now $T$ satisfies the conditions of Theorem \ref{thm:sharpFIO1}, and 
\[
\big\|\sup_{t\in\R}|T_{t}f|\big\|_{L^{p}(\Rn)}\eqsim \|T_{t_{0}}f\|_{L^{p}(\Rn)}.
\]
However, an inequality as in \eqref{eq:sharpFIO1} merely implies that $s\geq s(p)$, whereas $s(p)<d(p)+\frac{1}{p}$ for all $p\in[1,2]$. Instead, one has to take the dependence of the phase on the variable $t$ into account. 

We will now show that nontrivial dependence of the phase on $t$ is in fact the only additional assumption required to prove the desired sharpness bound.  The following result implies in particular that, under mild conditions on the symbol of the operator, the maximal function bounds in Theorems \ref{thm:maximalFIO1} and \ref{thm:maximalFIO2}, as well as in Corollary \ref{cor:pointwiseFIO1}, are essentially sharp for $1\leq p\leq 2$. 

\begin{theorem}\label{thm:sharpFIO2}
Let $p\in[1,\infty]$ and $s\in \R$. Suppose that there exist a $z_{0}\in\R^{n+1}$, a non-empty open conic set $V'\subseteq \Rn\setminus\{0\}$ and $c',C'>0$ such that $|a(z_{0},\eta)|\geq c'|\eta|^{m}$ and 
$\partial_{t}\Phi(x_{0},t,\eta)|_{t=t_{0}}\neq 0$ for all $\eta\in V'$ with $|\eta|\geq C'$. If there exists a $C\geq0$ such that
\begin{equation}\label{eq:sharpFIO2}
\big\|\sup_{t\in \R}|T_{t}f|\big\|_{L^{p}(\Rn)}\leq C\|f\|_{\HT^{s,p}_{FIO}(\Rn)}
\end{equation}
for all $f\in\Sw(\Rn)$, then $s\geq \frac{n-1}{2}(\frac{1}{p}-\frac{1}{2})+\frac{1}{p}+m$.
\end{theorem}
\begin{proof}
In the same way as in the proof of Theorem \ref{thm:sharpFIO1}, one may reduce to the case where $z_{0}=0$, $m=0$, and 
$\Phi(0,\eta)=0$ and $a(0,\eta)=1$ for all $\eta\in V'$. 
Fix a $\nu\in V'$. Then, by precomposing with the change of coordinates $y\mapsto \partial_{\eta x}^{2}\Phi(0,\nu)^{-1}y$, which is a diffeomorphism on $\Rn$ by the assumptions on $\Phi$ and therefore induces an automorphism of $\Hps$, we may suppose additionally that $\partial_{x\eta}^{2}\Phi(0,\nu)$ is the identity matrix. Finally, after rotating, we may assume that $\nu=e_{1}$, which will be of minor notational convenience. 

Let $\psi\in \Sw(\Rn)$ be such that $\wh{\psi}$ is non-negative and not identically zero, and such that $\wh{\psi}(\eta)=0$ if $|\eta|>\theta$, for some small parameter $\theta>0$ to be fixed later. Let $k\in\Z_{+}$, and let $f$ be as in \eqref{eq:fnudef}:
\[
f(x):=e^{i2^{k}x_{1}}\psi(2^{k}x_{1},2^{k/2}x')
\]
for $x=(x_{1},x')\in\R\times\R^{n-1}=\Rn$. By choosing $\theta$ sufficiently small, we may suppose that
\[
\supp(\wh{f}\,)\subseteq V'\cap \{\eta\in\Rn\mid 2^{k-1}\leq |\eta|\leq 2^{k+1},|\hat{\eta}-e_{1}|\leq 2^{-k/2}\}.
\]
Then $\|\wh{f}\,\|_{L^{1}(\Rn)}\eqsim 1$ and $\|f\|_{L^{p}(\Rn)}\eqsim 2^{-k\frac{n+1}{2p}}$, and in particular
\begin{equation}\label{eq:fnorm2}
\|f\|_{\Hps}\eqsim 2^{k(s+\frac{n-1}{2}(\frac{1}{2}-\frac{1}{p}))}\|f\|_{L^{p}(\Rn)}\eqsim 2^{k(s+\frac{n-1}{2}(\frac{1}{2}-\frac{1}{p})-\frac{n+1}{2p})},
\end{equation}
by \cite[Proposition 2.2]{Rozendaal22b}. Here the implicit constants depend on $\theta$, but this does not matter since we will only apply \eqref{eq:fnorm2} once we have already fixed $\theta$. Finally, there exist $c_{0},\delta>0$ such that
\begin{equation}\label{eq:Realfy}
\Real f(x)\geq c_{0}
\end{equation}
whenever $|x_{1}|\leq \delta 2^{-k}$ and $|x'|\leq \delta 2^{-k/2}$.

Next, set
\[
E_{0}:=\{y=(y_{1},y')\in\Rn\mid y_{1}=0, |y'|\leq \theta2^{-k/2}\}.
\]
Then, for every $y\in E_{0}$, the curve $\gamma_{y}$ defined by
\[
\partial_{\eta}\Phi(\gamma_{y}(t),t,e_{1})
=\partial_{\eta}\Phi(y,0,e_{1})
\]
is well defined and smooth in a neighborhood of zero, by the implicit function theorem. Moreover, by homogeneity,
\[
-e_{1}\cdot \gamma_{0}'(0)=e_{1}\cdot\partial_{t}\partial_{\eta}\Phi(0,0,e_{1})=e_{1}\cdot\partial_{\eta}(\partial_{t}\Phi)(0,0,e_{1})=\partial_{t}\Phi(0,0,e_{1})\neq 0.
\]
So in a small neighborhood of zero in $E_{0}$, at $t=0$ the curves $\gamma_{y}$ all point out of the plane where $x_{1}=0$, and in particular out of $E_{0}$. Hence, by shrinking $\theta$ even more, we may suppose that the curves $\gamma_{y}$ are mutually disjoint as $y$ ranges over $E_{0}$. Set
\[
E:=\cup_{y\in E_{0}}\gamma_{y}([-\theta,\theta]).
\]
Then, since the $(n-1)$-dimensional measure of $E_{0}$ is approximately $2^{-k\frac{n-1}{2}}$, the $n$-dimensional measure of $E$ is approximately $2^{-k\frac{n-1}{2}}$ as well. Again, here the implicit constants depend on $\theta$, but this will not matter once $\theta$ has been fixed.

By definition, for each $x\in E$ there exist $y\in E_{0}$ and $t\in [-\theta,\theta]$ such that $\partial_{\eta}\Phi(x,t,e_{1})=\partial_{\eta}\Phi(y,0,e_{1})$. Taylor approximation, combined with the assumptions on $\Phi$, show that the latter equals
\begin{align*}
&\partial_{\eta}\Phi(0,0,e_{1})+\partial_{x\eta}^{2}\Phi(0,0,e_{1})y+\sum_{\alpha\in\Z_{+}^{n},|\alpha|=2}\int_{0}^{1}(1-r)\partial_{x}^{\alpha}\partial_{\eta}\Phi(ry,0,e_{1})\ud r\cdot y^{\alpha}\\
&=y+\sum_{\alpha\in\Z_{+}^{n},|\alpha|=2}\int_{0}^{1}(1-r)\partial_{x}^{\alpha}\partial_{\eta}\Phi(ry,0,e_{1})\ud r\cdot y^{\alpha}.
\end{align*}
Since $y_{1}=0$ and $|y'|\leq \theta2^{-k/2}$, we obtain from this that the first coordinate of $\partial_{\eta}\Phi(x,t,e_{1})$ is bounded in absolute value by
\[
|y_{1}|+\sum_{\alpha\in\Z_{+}^{n},|\alpha|=2}\Big|\int_{0}^{1}(1-r)\partial_{x}^{\alpha}\partial_{\eta}\Phi(ry,0,e_{1})\ud r\Big| |y|^{\alpha}\lesssim \theta2^{-k}.
\]
On the other hand, the remaining coordinates are bounded by
\[
|y'|+\sum_{\alpha\in\Z_{+}^{n},|\alpha|=2}\Big|\int_{0}^{1}(1-r)\partial_{x}^{\alpha}\partial_{\eta}\Phi(ry,0,e_{1})\ud r\Big| |y|^{\alpha}\lesssim \theta2^{-k/2}.
\]
Thus, by \eqref{eq:Realfy}, we can choose $\theta$ small enough such that
\begin{equation}\label{eq:Realfx}
\Real f(\partial_{\eta}\Phi(x,t,e_{1}))\geq c_{0}
\end{equation}
for every $x\in E$ and a suitable $t\in[-\theta,\theta]$.

We can now derive the crucial estimate. Indeed, Lemma \ref{lem:flow} shows that
\[
|T_{t}f(x)-(2\pi)^{n}f(\partial_{\eta}\Phi(x,t,e_{1}))|\lesssim \|\wh{f}\,\|_{L^{1}(\Rn)}|(x,t)|\lesssim \theta
\]
for all $x\in E$ and $t\in[-\theta,\theta]$. So \eqref{eq:Realfx} implies that
\[
\Real T_{t}f(x)\geq (2\pi)^{n}\Real f(\partial_{\eta}\Phi(x,t,e_{1}))-|T_{t}f(x)-(2\pi)^{n}f(\partial_{\eta}\Phi(x,0,e_{1}))|\gtrsim 1
\]
for the appropriate $t\in[-\theta,\theta]$, where we possibly had to shrink $\theta$ again. Fix the resulting $\theta$.

Then the previous bound, combined with the estimate for $|E|$ as well as \eqref{eq:sharpFIO2} and \eqref{eq:fnorm2}, allows us to conclude the proof. Indeed, one has
\begin{align*}
2^{-k\frac{n-1}{2p}}&\eqsim |E|^{1/p}\lesssim \Big(\int_{E}\sup_{t\in\R}|T_{t}f(x)|^{p}\ud x\Big)^{1/p}\lesssim \|f\|_{\HT^{s,p}_{FIO}(\Rn)}\\
&\eqsim 2^{k(s+\frac{n-1}{2}(\frac{1}{2}-\frac{1}{p})-\frac{n+1}{2p})}.
\end{align*}
After rearranging and letting $k\to\infty$, this concludes the proof.
\end{proof}

\begin{remark}\label{rem:sharpintermediate}
Combined, Theorem \ref{thm:sharpFIO1} and \ref{thm:sharpFIO2} show that, under the assumptions of Theorem \ref{thm:sharpFIO2}, for general $p\in[1,\infty]$ one has 
\[
s\geq \begin{cases}
s(p)+\frac{1}{p}+m&\text{if }1\leq p\leq 2,\\
-s(p)+\frac{1}{p}+m&\text{if }2\leq p\leq \frac{2n}{n-1},\\
s(p)+m&\text{if }\frac{2n}{n-1}\leq p\leq \infty.
\end{cases}
\]
The same holds for the other sharpness statements in this section. We do not know whether the corresponding results in Section \ref{sec:maximal} can be improved for $2<p<\frac{2(n+1)}{n-1}$.
\end{remark}

\begin{remark}\label{rem:sharpLp}
Recall that, under the assumptions of Theorems \ref{thm:maximalFIO1} and \ref{thm:maximalFIO2}, the Sobolev embeddings in \eqref{eq:Sobolev} yield an inequality of the form
\begin{equation}\label{eq:sharpLp}
\big\|\sup_{t\in \R}|T_{t}f|\big\|_{L^{p}(\Rn)}\lesssim \|f\|_{\HT^{s,p}(\Rn)}
\end{equation}
for all $p\in[1,2]$, $s>2s(p)+\frac{1}{p}+m$ and $f\in \HT^{s,p}(\Rn)$. On the other hand, the arguments in the proof of Theorem \ref{thm:sharpFIO2} merely show that an inequality such as \eqref{eq:sharpLp} implies that $s>\frac{1}{p}+m$, due to the discrepancy between the $\Hps$ norm and the $\HT^{s,p}(\Rn)$ norm of the function $f$ in \eqref{eq:fnorm2}, in particular for $p< 2$.
\end{remark}

\vanish{
\begin{remark}\label{rem:cinematicsharp}
It is worth pointing out that if $T$ satisfies the cinematic curvature condition from \cite{Sogge91}, then the nontriviality condition on the phase in Theorem \ref{thm:sharphyp2} is satisfied (see e.g.~\cite[Section 8.1]{Sogge17}). This fact implicitly plays a role in the proofs of Propositions \ref{prop:sharpman1} and \ref{prop:sharpman2}.
\end{remark}
}

\subsection{Sharpness for hypersurfaces}\label{subsec:sharpsurf}

In this subsection we prove that the maximal function bounds in Sections \ref{subsec:hypersurf} and \ref{subsec:complex} are essentially sharp for suitable $p$.

We first prove sharpness in the case of Theorem \ref{thm:maximalhyp1}. Recall the definitions of $\A_{t}$ and $\Ma_{\Sigma,\psi}$ from \eqref{eq:defA} and \eqref{eq:defM}, for $\Sigma\subseteq\Rn$ a hypersurface, $\psi\in C^{\infty}_{c}(\Rn)$ and $t>0$. 
The following result shows that the maximal function estimates in Theorem \ref{thm:maximalhyp1} are essentially sharp for $p\in[1,2]\cup [\frac{2(n+1)}{n-1},\infty)$ in all nontrivial cases, and in particular in the case of Theorem \ref{thm:maximalhypintro} (see also Remark \ref{rem:pinfty}).

\begin{theorem}\label{thm:sharphyp1}
Let $p\in[1,\infty]$ and $s\in\R$. Let $\Sigma\subseteq\Rn$ be a hypersurface with non-vanishing Gaussian curvature, and let $ \psi\in C^{\infty}_{c}(\Rn)$ not be identically zero. If there exist $t_{0},C_{t_{0}}>0$ such that
\begin{equation}\label{eq:sharphyp11}
\|\A_{t_{0}}f\|_{L^{p}(\Rn)}\leq C_{t_{0}}\|f\|_{\Hps}
\end{equation}
for all $f\in \Sw(\Rn)$, then $s\geq \frac{n-1}{2}(\frac{1}{2}-\frac{1}{p})-\frac{n-1}{2}$. Moreover, if there exists a $C\geq 0$ such that
\begin{equation}\label{eq:sharphyp12}
\|\Ma_{\Sigma,\psi}f\|_{L^{p}(\Rn)}\leq C\|f\|_{\Hps}
\end{equation}
for all $f\in\Sw(\Rn)$, then $s\geq \frac{n-1}{2}(\frac{1}{p}-\frac{1}{2})+\frac{1}{p}-\frac{n-1}{2}$.
\end{theorem}
\begin{proof}
We would like to apply Theorems \ref{thm:sharpFIO1} and \ref{thm:sharpFIO2}. This is complicated by the fact that, although one may express $\A_{t}$ in terms of a sum of operators as in \eqref{eq:Tstandard2}, the contributions from the summands may, a priori, cancel each other out. To deal with this issue we will microlocalize in regions of phase space where there is only a single summand. 
We will reason in a manner which can also be used for other results in this section. Throughout, we freely make use of the theory of Fourier integral operators as can be found in \cite[Chapter XXV]{Hormander09} (see also \cite[Chapter 6]{Sogge17}).

\subsubsection{Preliminary work}
Note that we can represent $\A_{1}$ as in \eqref{eq:Lagrangian}:
\[
\A_{1}f(x)=\int_{x-\Sigma}f(y)\psi(x-y)\ud\sigma(x-y)=\int_{\Rn}\int_{\R}e^{i\theta\Psi_{1}(x,y)}\rho_{0}(x,1,y)f(y)\ud\theta\ud y
\]
for all $f\in\Sw(\Rn)$ and $x\in\Rn$. Here $\Psi_{1}\in C^{\infty}(\Rn\times \Rn)$ is such that $x-\Sigma=\{y\in\Rn\mid \Psi_{1}(x,y)=0\}$ for all $x\in\Rn$, and $\rho_{0}(x,1,y)=(2\pi)^{-1}\psi(x-y)|\partial_{y}\Psi_{1}(x,y)|$ for $y\in x-\Sigma$. By dilation, one then has
\begin{equation}\label{eq:At}
\A_{t}f(x)=\int_{\Rn}\int_{\R}e^{i\theta\Psi_{t}(x,y)}\rho_{0}(x,t,y)f(y)\ud\theta\ud y
\end{equation}
for all $t>0$, $f\in \Sw(\Rn)$ and $x\in\Rn$, where $\Psi_{t}(x,y)=\Psi_{1}(x/t,y/t)$ and $\rho_{0}(x,t,y)=t^{-n}\rho_{0}(x/t,1,y/t)$. Moreover, the assumption that $\Sigma$ has non-vanishing Gaussian curvature is equivalent to \eqref{eq:rotation} (see e.g.~\cite[Section XI.3.1]{Stein93}), and $\rho_{0}(\cdot,t,\cdot)$ is not identically zero. It follows that each $\A_{t}$ is a Fourier integral operator of order $-(n-1)/2$ with canonical relation 
\[
\Ca_{t}=\{(x,\theta\partial_{x}\Psi_{t}(x,y),y,-\theta\partial_{y}\Psi_{t}(x,y))\mid \Psi_{t}(x,y)=0,\theta\in\R\setminus\{0\}\}.
\]
Also, \eqref{eq:rotation} just says that the Jacobian of 
\[
(y,\theta)\mapsto \big(\partial_{x}(\theta\Psi_{t}(x,y)), \theta\Psi_{t}(x,y)\big)
\]
does not vanish where $\Psi_{t}(x,y)=0$ and $\theta\neq 0$. Hence one can solve $\Psi_{t}(x,y)=0$ and $\theta\partial_{x}\Psi_{t}(x,y)=\xi$ with respect to $(y,\theta)$ and use $(x,\xi)$ as local coordinates on $\Ca_{t}$. This in turn implies that the projection from $\Ca_{t}$ onto the $(x,\xi)$ coordinates is a submersion. Similarly, one can use $(y,\eta)$ as local coordinates on $\Ca_{t}$, and the projection from $\Ca_{t}$ onto the $(y,\eta)$ coordinates is also a submersion.

Finally, since $\rho_{0}(\cdot,t,\cdot)$ is not identically zero, the principal symbol of $\A_{t}$ does not vanish at some point $c_{t}=(x_{t},\xi_{t},y_{t},\eta_{t})\in \Ca_{t}$.

\subsubsection{Proof of the first statement}

{Suppose that \eqref{eq:sharphyp11} holds. 
We will microlocalize by multiplying $\A_{t_{0}}$ from the right by a pseudodifferential operator. More precisely, let $S$ be a pseudodifferential operator of order zero on $\Rn$, with a principal symbol which is nonzero at $(y_{t_{0}},\eta_{t_{0}})$, and with microsupport contained in a small conic neighborhood of this point. Since $\Hps$ is invariant under pseudodifferential operators of order zero, one then has
\[
\|\A_{t_{0}}Sf\|_{L^{p}(\Rn)}\lesssim \|f\|_{\Hps}
\]
for all $f\in\Sw(\Rn)$. As noted above, the canonical relation of $\A_{t_{0}}$, which in turn equals that of $\A_{t_{0}}S$, can be locally parametrized near $c_{t_{0}}=(x_{t_{0}},\xi_{t_{0}},y_{t_{0}},\eta_{t_{0}})$ using the $(y,\eta)$ coordinates. Given that $S$ has microsupport near $(y_{t_{0}},\eta_{t_{0}})$, the canonical relation of $\A_{t_{0}}S$ is thus contained in a small conic neighborhood of $c_{t_{0}}$, and its principal symbol does not vanish at $c_{t_{0}}$. In fact, after multiplying on the right by a smooth cutoff, we may additionally suppose that $\A_{t_{0}}Sf=0$ if $f$ is supported outside a small neighborhood of $y_{t_{0}}$. }

It then follows (see \cite[Lemma A.2 and Remark A.4]{LiRoSoYa24}) that one can write $\A_{t_{0}}S=T_{t_{0}}S_{t_{0}}'+R_{t_{0}}$, where $R_{t_{0}}$ is a smoothing operator, $S_{t_{0}}'$ is a change of coordinates and $T_{t_{0}}$ is an operator as in \eqref{eq:Tstandard2}, with $m=-(n-1)/2$, albeit with the $(n+1)$-dimensional $z$ variable replaced by the $n$-dimensional $x$ variable. Moreover, $T_{t_{0}}$ satisfies the $n$-dimensional version of the conditions in Theorem \ref{thm:sharpFIO1}. Here we have used the projection conditions on $\Ca_{t_{0}}$.

Since $\|{S_{t_{0}}'}f\|_{\Hps}\eqsim \|f\|_{\Hps}$ for $f\in\Sw(\Rn)$ supported in a suitably small neighborhood of $y_{t_{0}}$, and because $R_{t_{0}}:\Hps\to L^{p}(\Rn)$, we have 
\[
\|T_{t_{0}}f\|_{L^{p}(\Rn)}\lesssim \|f\|_{\Hps}
\]
for $f\in\Sw(\Rn)$. Now Theorem \ref{thm:sharpFIO1} and Remark \ref{rem:singleFIO} imply that $s\geq \frac{n-1}{2}(\frac{1}{2}-\frac{1}{p})-\frac{n-1}{2}$. 

\subsubsection{More preliminary work} 
We will now prepare for the proof of the second statement.

Set $\A f(x,t):=\chi(t)\A_{t}f(x)$ for $f\in\Sw(\Rn)$, $x\in\Rn$ and $t\in\R$, where $\chi\in C^{\infty}_{c}(\R)$ satisfies $\supp(\chi)\subseteq[1,2]$ and does not vanish at some point $t_{0}\in (1,2)$. Then $\A$ is a Fourier integral operator of order $-\frac{n-1}{2}-\frac{1}{4}$ with canonical relation
\begin{equation}\label{eq:canon}
\Ca:=\{(x,t,\theta\partial_{x}\Psi_{t}(x,y),\theta\partial_{t}\Psi_{t}(x,y),y,\theta\partial_{y}\Psi_{t}(x,y))\mid  y\in x-\Sigma,\theta\in\R\setminus\{0\}\}.
\end{equation}
The projections onto the $(x,t)$ coordinate and onto the last two coordinates are submersions. This implies that, if the $(x,t)$ coordinate and the last two coordinates are fixed on $\Ca$, then so are the remaining coordinates. Moreover, the function $\chi(t_{0})\rho_{0}(\cdot,t_{0},\cdot)$ is not identically zero. Hence the principal symbol of $\A$ does not vanish at some point $c_{0}=(x_{0},t_{0},\xi_{0},\tau_{0},y_{0},\eta_{0})\in \Ca$. In fact, the principal symbol is nonzero in an open neighborhood of $c_{0}$ in $\Ca$, and the $y$-coordinate in this neighborhood may vary over an open neighborhood within the hypersurface $x_{0}-\Sigma$, by \eqref{eq:canon}. In particular, we may thus assume that $x_{0}\neq y_{0}$.

\subsubsection{Proof of the second statement}
Suppose that \eqref{eq:sharphyp12} holds. As in the proof of the first statement, we will microlocalize $\A$. However, unlike in the setting of the first statement, $\Ca$ cannot be locally parametrized using only its last two coordinates, so we also need to multiply $\A$ on the left by a pseudodifferential operator. On the other hand, the mixed norm implicit on the left-hand side of \eqref{eq:sharphyp12} is not invariant under general pseudodifferential operators, so we have to be slightly careful here.

Again let $S$ be a pseudodifferential operator of order zero with a principal symbol which is nonzero at $(y_{0},\eta_{0})$, and with microsupport contained in a small conic neighborhood of this point.  Let $\chi_{1}\in C^{\infty}_{c}(\R^{n+1})$ be nonzero at $(x_{0},t_{0})$ and supported in a small neighborhood of this point. Then
\[
\big\|\sup_{t\in\R}|\chi_{1}(\cdot,t)\A Sf(\cdot,t)|\big\|_{L^{p}(\Rn)}\lesssim \big\|\sup_{1\leq t\leq 2}|\chi(t)\A_{t}Sf|\big\|_{L^{p}(\Rn)}\lesssim \|f\|_{\Hps}
\]
for all $f\in\Sw(\Rn)$. Here we used \eqref{eq:sharphyp12}, that $\Hps$ is invariant under $S$, and that $\chi$ and $\chi_{1}$ are multiplication operators.
Moreover, since $\Ca$ is determined by its $(x,t)$ coordinate and its last two coordinates, it suffices to localize near $(x_{0},t_{0})$ and microlocalize near $(y_{0},\eta_{0})$, as we have done, to guarantee that the canonical relation of $\chi_{1}\A S$ is contained in a small conic neighborhood of $c_{0}$. Finally, after multiplying by another smooth cutoff, we may suppose that $\chi_{1}\A Sf=0$ if $f$ is supported outside a small neighborhood of $y_{0}$.

It then follows (see \cite[Lemma A.2 and Remark A.4]{LiRoSoYa24}) that $\chi_{1}\A S=TS'+R$, where $R$ is a smoothing operator, $S'$ is a change of coordinates and $T$ is an operator as in \eqref{eq:Tstandard2}, with $m=-(n-1)/2$, satisfying the conditions of Theorem \ref{thm:sharpFIO1}. One again has $\|{S'}f\|_{\Hps}\eqsim \|f\|_{\Hps}$ for $f$ supported in a suitably small neighborhood of $y_{0}$. Moreover, as in the proof of Theorem \ref{thm:maximalFIO1}, one can show that $\|\sup_{t\in\R}|Rf(\cdot,t)|\|_{L^{p}(\Rn)}\lesssim \|f\|_{\Hps}$ for all $f\in\Hps$. Hence
\begin{equation}\label{eq:sharphyp13}
\big\|\sup_{t\in\R}|T_{t}f|\big\|_{L^{p}(\Rn)}\lesssim \|f\|_{\Hps}
\end{equation}
for $f\in\Sw(\Rn)$. It thus remains to show that the phase $\Phi$ of satisfies the nontriviality condition in Theorem \ref{thm:maximalFIO2}.

To this end 
first note that, by the expression for $T$ in \eqref{eq:Tstandard2}, the canonical relation of $T$ is contained in
\begin{equation}\label{eq:canon1}
\Ca_{1}:=\{(x,t,\partial_{x}\Phi(x,t,\eta),\partial_{t}\Phi(x,t,\eta),\partial_{\eta}\Phi(x,t,\eta),\eta)\mid (x,t,\eta)\in V\}.
\end{equation}
On the other hand, since $R$ is smoothing, the canonical relation of $TS'$ coincides with that of $\chi_{1}\A S$, which in turn is contained in a small neighborhood of $c_{0}$ in $\Ca$. Moreover, since $S'$ is a change of coordinates on $\Rn$, this means that the canonical relation of $T$ is contained in 
\begin{equation}\label{eq:canon2}
\Ca_{2}:=\{(x,t,\xi,\tau,\ka(y),(\partial_{\ka}(y))_{*}\eta)\mid  (x,t,\xi,\tau,y,\eta)\in\Ca\}.
\end{equation}
Here $\ka$ is a diffeomorphism on $\Rn$ and $(\partial_{\ka}(y))_{*}$ is the induced map on the cotangent space $T^{*}_{y}\Rn$. Let $z_{0}\in\R^{n+1}$ and $V'\subseteq \Rn\setminus\{0\}$ be as in Theorem \ref{thm:sharpFIO2}, and let $\eta'\in V'$. Then, by comparing \eqref{eq:canon}, \eqref{eq:canon1} and \eqref{eq:canon2}, we see that there exist $x\in\Rn$, $y\in x-\Sigma$ and $\theta\neq0$ such that $\partial_{t}\Phi(z_{0},\eta')=\theta\partial_{t}\Psi_{t}(x,y)$. Since $x_{0}\neq y_{0}$, we may also assume that $x\neq y$. On the other hand, since $\A_{1}$ is translation invariant, one has $\Psi_{t}(x,y)=\Psi(\frac{x-y}{t})$ for some $\Psi\in C^{\infty}(\Rn)$. So
\[
\partial_{t}\Phi(z_{0},\eta')=-\tfrac{1}{t^{2}}\theta(\nabla\Psi)(\tfrac{x-y}{t})\cdot(x-y),
\]
and the latter is nonzero by the assumption that $\Sigma$ has non-vanishing Gaussian curvature. Now Theorem \ref{thm:sharpFIO2} and \eqref{eq:sharphyp13} conclude the proof.
\end{proof}

Next, we consider Theorem \ref{thm:maximalhyp2}, which we show to be essentially sharp for all $p\in[1,2]\cup [\frac{2(n+1)}{n-1},\infty)$ under natural nontriviality assumptions (see also Remark \ref{rem:pinfty} for $p=\infty$). We use notation as in Theorem \ref{thm:maximalhyp2}.

\begin{theorem}\label{thm:sharphyp2}
Let $p\in[1,\infty]$ and $s\in\R$. Let $\Psi\in C^{\infty}(\R^{n+1}\times\Rn)$ be such that \eqref{it:rotational} holds, and let $\rho\in C^{\infty}_{c}(\R^{n+1}\times\Rn)$, $(x_{0},t_{0})\in\R^{n+1}$ and $y_{0}\in\Sigma_{x_{0},t_{0}}$ be such that $\rho(x_{0},t_{0},y_{0})\neq 0$. If there exists a $C_{t_{0}}\geq0$ such that
\[
\|T_{t_{0}}f\|_{L^{p}(\Rn)}\leq C_{t_{0}}\|f\|_{\Hps}
\]
for all $f\in \Sw(\Rn)$, then $s\geq \frac{n-1}{2}(\frac{1}{2}-\frac{1}{p})-\frac{n-1}{2}$. Moreover, if $\partial_{t}\Psi_{t}(x_{0},y_{0})|_{t=t_{0}}\neq 0$ and if there exists a $C\geq0$ such that 
\[
\big\|\sup_{t\in\R}|T_{t}f|\big\|_{L^{p}(\Rn)}\leq C\|f\|_{\Hps}
\]
for all $f\in\Sw(\Rn)$, then $s\geq \frac{n-1}{2}(\frac{1}{p}-\frac{1}{2})+\frac{1}{p}-\frac{n-1}{2}$.
\end{theorem}

\begin{proof}
The proof is analogous to that of Theorem \ref{thm:sharphyp1}, but it requires less work. Indeed, 
each $T_{t}$ can be expressed as in \eqref{eq:At}, and \eqref{it:rotational} is satisfied by assumption. Then the same arguments as in the proof of Theorem \ref{thm:sharphyp1} yield the first statement.

The proof of the second statement is also analogous to that of Theorem \ref{thm:sharphyp1}, but it requires slightly less work because $\partial_{t}\Psi_{t}(x_{0},y_{0})|_{t=t_{0}}\neq 0$ holds by assumption.
\end{proof}

\begin{remark}\label{rem:partialdef}
As in Theorem \ref{thm:maximalhyp2} (see Remark \ref{rem:specificsurface}), one does not need $\Psi$ to satisfy \eqref{it:rotational} on all of $\R^{n+1}\times \Rn$, only in a neighborhood of the support of $\rho$.

Also, one may suppose that $\partial_{t}\Psi_{t}(x_{0},y_{0})|_{t=t_{0}}\neq 0$ in Theorem \ref{thm:sharphyp2} whenever condition \eqref{it:cinematic} in Theorem \ref{thm:maximalhyp2} is satisfied, by the same argument as in Remark \ref{rem:cinematic}.
\end{remark}
We also obtain sharpness for all $\alpha\in\C$ and $p\in [1,2]\cup  [\frac{2(n+1)}{n-1},\infty)$ in Proposition \ref{prop:complexsphere} (see Remark \ref{rem:pinfty} for $p=\infty$), as a consequence of the following result. Recall the definition of $\Ma_{t}^{\alpha}$ from \eqref{eq:complexmean}, for $\alpha\in\C$ and $t>0$.

\begin{proposition}\label{prop:sharpcomplex}
Let $p\in[1,\infty]$, $s\in\R$ and $\alpha\in\C$. If there exist $t_{0},C_{t_{0}}>0$ such that
\begin{equation}\label{eq:sharpcomplex1}
\|\Ma_{t_{0}}^{\alpha}f\|_{L^{p}(\Rn)}\leq C_{t_{0}}\|f\|_{\Hps}
\end{equation}
for all $f\in \Sw(\Rn)$, then $s\geq \frac{n-1}{2}(\frac{1}{2}-\frac{1}{p})-\frac{n-1}{2}-\Real(\alpha)$. Moreover, if there exist a non-empty open interval $I\subseteq (0,\infty)$ and a $C\geq0$ such that
\begin{equation}\label{eq:sharpcomplex2}
\big\|\sup_{t\in I}|\Ma_{t}^{\alpha}f|\big\|_{L^{p}(\Rn)}\leq C\|f\|_{\Hps}
\end{equation}
for all $f\in\Sw(\Rn)$, then $s\geq \frac{n-1}{2}(\frac{1}{p}-\frac{1}{2})+\frac{1}{p}-\frac{n-1}{2}-\Real(\alpha)$.
\end{proposition}
\begin{proof}
The scheme of the proof is similar to that of Theorem \ref{thm:sharphyp1}, although some care is required to show that the principal symbol of each $\Ma_{t}^{\alpha}$ is nonzero somewhere.

Recall from \eqref{eq:malpha} that 
\[
m_{\alpha}(\xi)=a_{\alpha}^{0}(\xi)+e^{i|\xi|}a_{\alpha}^{+}(\xi)+e^{-i|\xi|}a^{-}_{\alpha}(\xi)
\]
for all $\xi\in\Rn$, where $a_{\alpha}^{0}\in C^{\infty}_{c}(\Rn)$ and $a_{\alpha}^{+},a_{\alpha}^{-}\in S^{-\frac{n-1}{2}-\Real(\alpha)}(\Rn)$. In fact, it follows from asymptotics of Bessel functions (see \cite[Section 7.21]{Watson95}) that $a^{+}_{\alpha}$ and $a^{-}_{\alpha}$ are elliptic. For $t>0$ and $\xi\in\R$, set 
\[
T_{t}^{+}(\xi):=e^{it|\xi|}a_{\alpha}^{+}(t\xi)\ \text{ and }\ T_{t}^{-}(\xi):=e^{-it|\xi|}a_{\alpha}^{-}(t\xi)
\]
Note that
\[
\Ca^{+}_{t}:=\{(x,\eta,x+ t\hat{\eta},\eta)\mid x\in\Rn,\eta\in\Rn\setminus\{0\}\}
\]
is the canonical relation of $T_{t}^{+}(D)$, and 
\[
\Ca^{-}_{t}:=\{(x,\eta,x-t\hat{\eta},\eta)\mid x\in\Rn,\eta\in\Rn\setminus\{0\}\}
\]
is the canonical relation of $T^{-}_{t}(D)$. 

Now, let $t_{0}>0$ be as in \eqref{eq:sharpcomplex1}. Fix $x_{0}\in\Rn$, and let $S$ be a pseudodifferential operator of order zero, with a principal symbol which is nonzero at $(x_{0}+t_{0}e_{1},e_{1})$, and with microsupport contained in a small neighborhoods of this point. Here $e_{1}\in\Rn$ is the first basis vector. Then, by construction and because $\Ca_{t}^{+}$ can be parametrized using its last two coordinates, the principal symbol of $T_{t_{0}}^{+}(D)S$ is nonzero at $(x_{0},e_{1},x_{0}+t_{0} e_{1},e_{1})$, and the canonical relation of $T_{t_{0}}^{+}(D)S$ is contained in a small conic neighborhood in $\Ca^{+}_{t}$ of this point. Moreover, the canonical relation of $T_{t_{0}}^{-}(D)S$ is empty, so $T_{t_{0}}^{-}(D)S$ is a smoothing operator.

By \eqref{eq:sharpcomplex1} and because $\Hps$ is invariant under pseudodifferential operators of order zero, one has $\|\Ma_{t_{0}}^{\alpha}Sf\|_{L^{p}(\Rn)}\lesssim \|f\|_{\Hps}$ for all $f\in\Sw(\Rn)$. In turn, since $a_{\alpha}^{0}(t_{0}D)$ and $T_{t_{0}}^{-}(D)S$ are smoothing operators, this implies that
\[
\|T_{t_{0}}^{+}(D)Sf\|_{L^{p}(\Rn)}\lesssim \|f\|_{\Hps}
\]
for $f\in\Sw(\Rn)$. Moreover, one can again write $T_{t_{0}}^{+}(D)S$ as the sum of an operator as in Theorem \ref{thm:sharpFIO1} precomposed with a change of coordinates, and a harmless smoothing operator. So Theorem \ref{thm:sharpFIO1} shows that $s\geq s(p)-\frac{n-1}{2}-\Real(\alpha)$ if $p\geq 2$.

Next, let $I\subseteq(0,\infty)$ be a non-empty open interval such that \eqref{eq:sharpcomplex2} holds, and fix $t_{0}\in I$ and $x_{0}\in\Rn$. Let $S$ be as before, and let $\chi\in C^{\infty}_{c}(\R^{n+1})$ be a smooth cutoff which localizes to a small neighborhood of $(x_{0},t_{0})$. Then \eqref{eq:sharpcomplex2} implies that
\[
\big\|\sup_{t\in \R}|\chi(\cdot,t)\Ma_{t}^{\alpha}Sf(\cdot)|\big\|_{L^{p}(\Rn)}\lesssim \|f\|_{\Hps}
\]
for all $f\in\Sw(\Rn)$.  Set $T^{+}f(x,t):=T_{t}^{+}(D)f(x)$ and $T^{-}f(x,t):=T^{-}_{t}(D)f(x)$ for $x\in\Rn$ and $t>0$. Then $T^{-}S$ is a smoothing operator, as is the term corresponding to $a_{\alpha}^{0}$. It thus follows that 
\[
\big\|\sup_{t\in \R}|{\chi(\cdot,t)T^{+}Sf(\cdot,t)}|\big\|_{L^{p}(\Rn)}\lesssim \|f\|_{\Hps}
\]
for all $f\in\Sw(\Rn)$. Since the the canonical relation of $\chi T^{+}S$ is contained in 
\[
\{(x,t,\eta,-|\eta|,x-t\hat{\eta},\eta)\mid (x,t)\in\R^{n+1},\eta\in\Rn\setminus\{0\}\}
\]
and its principal symbol does not vanish at $(x_{0},t_{0},\eta,-|\eta|,x_{0}-t_{0}e_{1},\eta)$, a similar argument as in the second half of the proof of Theorem \ref{thm:sharphyp1}, relying on Theorem \ref{thm:sharpFIO2}, shows that $s\geq s(p)+\frac{1}{p}-\frac{n-1}{2}-\Real(\alpha)$ if $p\leq 2$, as required.
\end{proof}

\subsection{Sharpness on manifolds}\label{subsec:sharpman}

Finally, we show that the results in Section \ref{subsec:manifold} are essentially sharp for suitable values of $p\in[1,\infty]$. Fix an $n$-dimensional compact Riemannian manifold $(M,g)$ without boundary, and let $\Delta_{g}$ be the associated Laplace--Beltrami operator. 

We first observe that the maximal function estimate in Proposition \ref{prop:maximalFIO3} (see also Remark \ref{rem:pinfty}) is essentially sharp for all $p\in[1,2]\cup[\frac{2(n+1)}{n-1},\infty)$.

\begin{proposition}\label{prop:sharpman1}
Let $p\in[1,\infty]$ and $s\in\R$. If there exist $t_{0}\in\R$ and $C_{t_{0}}\geq0$ such that
\[
\|e^{it_{0}\sqrt{-\Delta_{g}}}f\|_{L^{p}(M)}\leq C_{t_{0}}\|f\|_{\HpsM}
\]
for all $f\in \Da(M,\Omega_{1/2})$, then $s\geq \frac{n-1}{2}(\frac{1}{2}-\frac{1}{p})$. Moreover, if there exist a non-empty open interval $I\subseteq\R$ and a $C\geq0$ such that
\begin{equation}\label{eq:sharpman12}
\big\|\sup_{t\in I}|e^{it\sqrt{-\Delta_{g}}}f|\big\|_{L^{p}(M)}\leq C\|f\|_{\HpsM}
\end{equation}
for all $f\in\Da(M,\Omega_{1/2})$, then $s\geq \frac{n-1}{2}(\frac{1}{p}-\frac{1}{2})+\frac{1}{p}$.
\end{proposition}
\begin{proof}
For the first statement one can reduce to Theorem \ref{thm:sharpFIO1} by working in local coordinates. Alternatively, one can reason as follows. By replacing $f$ by $e^{-it_{0}\sqrt{-\Delta_{g}}}f$ and using the invariance of $\HpsM$ under $e^{-it_{0}\sqrt{-\Delta_{g}}}$, we may suppose that $t_{0}=0$. But then the first statement follows from the fact that the Sobolev embeddings in \eqref{eq:SobolevM} are sharp (which in turn is a consequence of e.g.~Theorem \ref{thm:sharpFIO1}).

For the second statement we do work explicitly in local coordinates. More precisely, let $K$ be a collection of coordinate charts as in Definition \ref{def:HpFIOM}. Fix a $\ka\in K$, $\ka:U_{\ka}\to\ka(U_{\ka})\subseteq\Rn$, and the associated $\psi_{\ka}$ from \eqref{eq:partitionunit}. We may suppose that $\psi_{\ka}\equiv 1$ on an open subset $U\subseteq U_{\ka}$. Now let $0\neq \psi\in C^{\infty}(M)$ be such that $\supp(\psi)\subseteq U$. Set $\psit:=\psi\circ \ka^{-1}$ and 
\[
T_{t}f(x):=\ka_{*}\big(\psi_{\ka} e^{it\sqrt{-\Delta_{g}}}\ka^{*}(\psit f)\big)(x),
\]
for $f\in\Sw(\Rn)$, $t\in\R$ and $x\in\Rn$. Let $\chi\in C^{\infty}_{c}(\R)$ satisfy $\supp(\chi)\subseteq I$ while not being identically zero, and set $Tf(x,t):=\chi(t)T_{t}f(x)$. Then, by \eqref{eq:sharpman12} and Definition \ref{def:HpFIOM}, we have
\begin{equation}\label{eq:maxlocal}
\big\|\sup_{t\in \R}|Tf(\cdot,t)|\big\|_{L^{p}(\Rn)}\lesssim \|\psit f\|_{\Hps}\lesssim \|f\|_{\Hps}
\end{equation}
for all $f\in\Sw(\Rn)$. 

Next, one can argue as in the proof of Theorem \ref{thm:maximalhyp1} to express $T$ microlocally in terms of an operator as in \eqref{eq:Tstandard2}, with $m=0$, after which one can combine \eqref{eq:maxlocal} with Theorems \ref{thm:sharpFIO1} and \ref{thm:sharpFIO2}. This is possible because $T$ is a Fourier integral operator of order $-1/4$, with a principal symbol which does not vanish identically, and with a canonical relation of the form
\begin{equation}\label{eq:canman}
\Ca:=\{(x,t,\xi,p(x,\xi),y,\eta)\mid (x,\xi)=\chi_{t}(y,\eta)\}.
\end{equation}
Here $\chi_{t}$ is the bicharacteristic flow at time $t$ associated with $\sqrt{g}$ in local coordinates, and $p(x,\xi)>0$ is the principal symbol of $\sqrt{-\Delta_{g}}$ in local coordinates. These statements about $T$ in turn follow by expressing the canonical relation of $(e^{it\sqrt{-\Delta_{g}}})_{t\in\R}$ (see e.g.~\cite[Theorem 29.1.1]{Hormander09}) in local coordinates. In particular, as in the proof of Theorem \ref{thm:sharphyp1}, the fact that the symbol $p$ is nonzero implies that the resulting operator satisfies the conditions of Theorem \ref{thm:sharpFIO2}.
\end{proof}

Finally, the maximal function estimate in Proposition \ref{prop:geodesicmax} is also essentially sharp for all $p\in[1,2]\cup[\frac{2(n+1)}{n-1},\infty)$ (see Remark \ref{rem:pinfty} for $p=\infty$). We use notation 
as in \eqref{eq:FIOgeodesic}.

\begin{proposition}\label{prop:sharpman2}
Let $p\in[1,\infty]$ and $s\in\R$. If there exist $t_{0}\in(0,\inj(M))$ and  $C_{t_{0}}\geq0$ such that
\[
\|\A_{t_{0},g}f\|_{L^{p}(M)}\leq C_{t_{0}}\|f\|_{\HpsM}
\]
for all $f\in \Da(M,\Omega_{1/2})$, then $s\geq \frac{n-1}{2}(\frac{1}{2}-\frac{1}{p})-\frac{n-1}{2}$. Moreover, if  there exist a non-empty open interval $I\subseteq(0,\inj (M))$ and a $C\geq0$ such that
\[
\big\|\sup_{t\in I}|\A_{t,g}f|\big\|_{L^{p}(M)}\leq C\|f\|_{\HpsM}
\]
for all $f\in\Da(M,\Omega_{1/2})$, then $s\geq \frac{n-1}{2}(\frac{1}{p}-\frac{1}{2})+\frac{1}{p}-\frac{n-1}{2}$.
\end{proposition}
\begin{proof}
Here one can argue in two ways. 

Firstly, by working in local coordinates as in the proof of Proposition \ref{prop:sharpman1}, one can reduce to an application of Theorems \ref{thm:sharpFIO1} and \ref{thm:sharpFIO2}. This is possible for the first statement because $\A_{t_{0},g}$ is a Fourier integral operator of order $-(n-1)/2$, associated with a local canonical graph and with a non-vanishing principal symbol. For the second statement one can also use that the canonical relation of the operator $T$, given by $Tf(x,t):=\chi(t)\A_{t,g}f(x)$ for $f\in \Da(M,\Omega_{1/2})$, $x\in M$, $t\in\R$ and $0\neq \chi\in C^{\infty}_{c}(\R)$ such that $\supp(\chi)\subseteq I$, is contained in the union of the canonical relation in \eqref{eq:canman} and the corresponding one with $p(x,\xi)$ replaced by $-p(x,\xi)$.

Alternatively, as in the proof of Proposition \ref{prop:geodesicmax}, one can observe that, in geodesic normal coordinate charts, each $\A_{t,g}$ can be expressed as in \eqref{eq:hypvarop} for some $\rho\in C^{\infty}_{c}(\R^{n+1}\times \Rn)$ which is not identically zero, and a defining function $\Psi\in C^{\infty}(\R^{n+1}\times \Rn)$ satisfying the conditions \eqref{it:rotational} and \eqref{it:cinematic} introduced there. Then Theorem \ref{thm:sharphyp2} (see also Remark \ref{rem:partialdef}) concludes the proof.
\end{proof}

\section*{Acknowledgments}

The authors would like to thank the referee for carefully reading the manuscript and for several very helpful suggestions.

\bibliographystyle{plain}
\bibliography{Bibliography}

\begin{thebibliography}{10}

\bibitem{BeHiSo20}
David Beltran, Jonathan Hickman, and Christopher~D. Sogge.
\newblock Variable coefficient {W}olff-type inequalities and sharp local
  smoothing estimates for wave equations on manifolds.
\newblock {\em Anal. PDE}, 13(2):403--433, 2020.

\bibitem{BeHiSo21}
David Beltran, Jonathan Hickman, and Christopher~D. Sogge.
\newblock Sharp local smoothing estimates for {F}ourier integral operators.
\newblock In {\em Geometric aspects of harmonic analysis}, volume~45 of {\em
  Springer INdAM Ser.}, pages 29--105. Springer, Cham, [2021] \copyright 2021.

\bibitem{BeRaSa19}
David Beltran, Jo\~{a}o~Pedro Ramos, and Olli Saari.
\newblock Regularity of fractional maximal functions through {F}ourier
  multipliers.
\newblock {\em J. Funct. Anal.}, 276(6):1875--1892, 2019.

\bibitem{Bourgain86b}
Jean Bourgain.
\newblock Averages in the plane over convex curves and maximal operators.
\newblock {\em J. Anal. Math.}, 47:69--85, 1986.

\bibitem{Bourgain16}
Jean Bourgain.
\newblock A note on the {S}chr\"{o}dinger maximal function.
\newblock {\em J. Anal. Math.}, 130:393--396, 2016.

\bibitem{Bourgain-Demeter15}
Jean Bourgain and Ciprian Demeter.
\newblock The proof of the {$l^2$} decoupling conjecture.
\newblock {\em Ann. of Math. (2)}, 182(1):351--389, 2015.

\bibitem{Bourgain-Demeter17}
Jean Bourgain and Ciprian Demeter.
\newblock Decouplings for curves and hypersurfaces with nonzero {G}aussian
  curvature.
\newblock {\em J. Anal. Math.}, 133:279--311, 2017.

\bibitem{ChLeLi26}
Chu-hee Cho, Sanghyuk Lee, and Wenjuan Li.
\newblock Endpoint estimates for maximal operators associated to the wave
  equation.
\newblock {\em J. Funct. Anal.}, 290(6):Paper No. 111308, 2026.

\bibitem{Cowling83b}
Michael Cowling.
\newblock Pointwise behavior of solutions to {S}chr\"odinger equations.
\newblock In {\em Harmonic analysis ({C}ortona, 1982)}, volume 992 of {\em
  Lecture Notes in Math.}, pages 83--90. Springer, Berlin, 1983.

\bibitem{DosSantosFerreira-Staubach14}
David Dos Santos~Ferreira and Wolfgang Staubach.
\newblock Global and local regularity of {F}ourier integral operators on
  weighted and unweighted spaces.
\newblock {\em Mem. Amer. Math. Soc.}, 229(1074):xiv+65, 2014.

\bibitem{DuGuLi17}
Xiumin Du, Larry Guth, and Xiaochun Li.
\newblock A sharp {S}chr\"{o}dinger maximal estimate in {$\Bbb R^2$}.
\newblock {\em Ann. of Math. (2)}, 186(2):607--640, 2017.

\bibitem{Du-Zhang19}
Xiumin Du and Ruixiang Zhang.
\newblock Sharp {$L^2$} estimates of the {S}chr\"{o}dinger maximal function in
  higher dimensions.
\newblock {\em Ann. of Math. (2)}, 189(3):837--861, 2019.

\bibitem{FaLiRoSo23}
Zhijie Fan, Naijia Liu, Jan Rozendaal, and Liang Song.
\newblock Characterizations of the {H}ardy space {$\mathcal H_{\rm FIO}^1(\Bbb
  R^n)$} for {F}ourier integral operators.
\newblock {\em Studia Math.}, 270(2):175--207, 2023.

\bibitem{GaLiMiXi23}
Chuanwei Gao, Bochen Liu, Changxing Miao, and Yakun Xi.
\newblock Square function estimates and local smoothing for {Fourier} integral
  operators.
\newblock {\em Proc. Lond. Math. Soc. (3)}, 126(6):1923--1960, 2023.

\bibitem{GuWaZh20}
Larry Guth, Hong Wang, and Ruixiang Zhang.
\newblock A sharp square function estimate for the cone in {$\Bbb {R}^3$}.
\newblock {\em Ann. of Math. (2)}, 192(2):551--581, 2020.

\bibitem{HaPoRo20}
Andrew Hassell, Pierre Portal, and Jan Rozendaal.
\newblock Off-singularity bounds and {H}ardy spaces for {F}ourier integral
  operators.
\newblock {\em Trans. Amer. Math. Soc.}, 373(8):5773--5832, 2020.

\bibitem{HaPoRoYu23}
Andrew Hassell, Pierre Portal, Jan Rozendaal, and Po-Lam Yung.
\newblock Function spaces for decoupling.
\newblock To appear in J. London Math. Soc. Preprint available at
  \url{https://arxiv.org/abs/2302.12701}, 2023.

\bibitem{Hickman23}
Jonathan Hickman.
\newblock Pointwise convergence for the {S}chr\"odinger equation [{\it after}
  {X}iumin {D}u and {R}uixiang {Z}hang].
\newblock {\em Ast\'erisque}, (446):Exp. No. 1205, 285--363, 2023.

\bibitem{Hormander09}
Lars H\"{o}rmander.
\newblock {\em The analysis of linear partial differential operators. {IV}}.
\newblock Classics in Mathematics. Springer-Verlag, Berlin, 2009.
\newblock Fourier integral operators, Reprint of the 1994 edition.

\bibitem{Lacey19}
Michael~T. Lacey.
\newblock Sparse bounds for spherical maximal functions.
\newblock {\em J. Anal. Math.}, 139(2):613--635, 2019.

\bibitem{Lee03}
Sanghyuk Lee.
\newblock Endpoint estimates for the circular maximal function.
\newblock {\em Proc. Amer. Math. Soc.}, 131(5):1433--1442, 2003.

\bibitem{LiRoSoYa24}
Naijia Liu, Jan Rozendaal, Liang Song, and Lixin Yan.
\newblock Local smoothing and {H}ardy spaces for {F}ourier integral operators
  on manifolds.
\newblock {\em J. Funct. Anal.}, 286(2):Paper No. 110221, 72 pp., 2024.

\bibitem{LiShSoYa25}
Naijia Liu, Minxing Shen, Liang Song, and Lixin Yan.
\newblock {$L^p$} bounds for {S}tein's spherical maximal operators.
\newblock {\em Math. Ann.}, 390(4):5235--5255, 2024.

\bibitem{MiYaZh17}
Changxing Miao, Jianwei Yang, and Jiqiang Zheng.
\newblock On local smoothing problems and {S}tein's maximal spherical means.
\newblock {\em Proc. Amer. Math. Soc.}, 145(10):4269--4282, 2017.

\bibitem{MoSeSo92}
Gerd Mockenhaupt, Andreas Seeger, and Christopher~D. Sogge.
\newblock Wave front sets, local smoothing and {B}ourgain's circular maximal
  theorem.
\newblock {\em Ann. of Math. (2)}, 136(1):207--218, 1992.

\bibitem{MoSeSo93}
Gerd Mockenhaupt, Andreas Seeger, and Christopher~D. Sogge.
\newblock Local smoothing of {F}ourier integral operators and
  {C}arleson-{S}j\"{o}lin estimates.
\newblock {\em J. Amer. Math. Soc.}, 6(1):65--130, 1993.

\bibitem{Phong-Stein86}
Duong~H. Phong and Elias~M. Stein.
\newblock Hilbert integrals, singular integrals, and {R}adon transforms. {I}.
\newblock {\em Acta Math.}, 157(1-2):99--157, 1986.

\bibitem{Rogers-Villarroya08}
Keith~M. Rogers and Paco Villarroya.
\newblock Sharp estimates for maximal operators associated to the wave
  equation.
\newblock {\em Ark. Mat.}, 46(1):143--151, 2008.

\bibitem{Roos-Seeger23}
Joris Roos and Andreas Seeger.
\newblock Spherical maximal functions and fractal dimensions of dilation sets.
\newblock {\em Amer. J. Math.}, 145(4):1077--1110, 2023.

\bibitem{Rozendaal21}
Jan Rozendaal.
\newblock Characterizations of {H}ardy spaces for {F}ourier integral operators.
\newblock {\em Rev. Mat. Iberoam.}, 37(5):1717--1745, 2021.

\bibitem{Rozendaal22b}
Jan Rozendaal.
\newblock Local smoothing and {H}ardy spaces for {F}ourier integral operators.
\newblock {\em J. Funct. Anal.}, 283(12):Paper No. 109721, 22 pp., 2022.

\bibitem{Schlag97}
Wilhelm Schlag.
\newblock A generalization of {B}ourgain's circular maximal theorem.
\newblock {\em J. Amer. Math. Soc.}, 10(1):103--122, 1997.

\bibitem{Schlag-Sogge97}
Wilhelm Schlag and Christopher~D. Sogge.
\newblock Local smoothing estimates related to the circular maximal theorem.
\newblock {\em Math. Res. Lett.}, 4(1):1--15, 1997.

\bibitem{SeSoSt91}
Andreas Seeger, Christopher~D. Sogge, and Elias~M. Stein.
\newblock Regularity properties of {F}ourier integral operators.
\newblock {\em Ann. of Math. (2)}, 134(2):231--251, 1991.

\bibitem{Smith98a}
Hart~F. Smith.
\newblock A {H}ardy space for {F}ourier integral operators.
\newblock {\em J. Geom. Anal.}, 8(4):629--653, 1998.

\bibitem{Sogge91}
Christopher~D. Sogge.
\newblock Propagation of singularities and maximal functions in the plane.
\newblock {\em Invent. Math.}, 104(2):349--376, 1991.

\bibitem{Sogge17}
Christopher~D. Sogge.
\newblock {\em Fourier integrals in classical analysis}, volume 210 of {\em
  Cambridge Tracts in Mathematics}.
\newblock Cambridge University Press, Cambridge, second edition, 2017.

\bibitem{Sogge-Stein90}
Christopher~D. Sogge and Elias~M. Stein.
\newblock Averages over hypersurfaces. {S}moothness of generalized {R}adon
  transforms.
\newblock {\em J. Analyse Math.}, 54:165--188, 1990.

\bibitem{Stein76}
Elias~M. Stein.
\newblock Maximal functions. {I}. {S}pherical means.
\newblock {\em Proc. Nat. Acad. Sci. U.S.A.}, 73(7):2174--2175, 1976.

\bibitem{Stein93}
Elias~M. Stein.
\newblock {\em Harmonic analysis: real-variable methods, orthogonality, and
  oscillatory integrals}, volume~43 of {\em Princeton Mathematical Series}.
\newblock Princeton University Press, Princeton, NJ, 1993.
\newblock With the assistance of Timothy S. Murphy, Monographs in Harmonic
  Analysis, III.

\bibitem{Stein-Weiss71}
Elias~M. Stein and Guido Weiss.
\newblock {\em Introduction to {F}ourier analysis on {E}uclidean spaces}.
\newblock Princeton Mathematical Series, No. 32. Princeton University Press,
  Princeton, N.J., 1971.

\bibitem{Walther99}
Bj\"orn~G. Walther.
\newblock Some {$L^p(L^\infty)$}- and {$L^2(L^2)$}-estimates for oscillatory
  {F}ourier transforms.
\newblock In {\em Analysis of divergence ({O}rono, {ME}, 1997)}, Appl. Numer.
  Harmon. Anal., pages 213--231. Birkh\"auser Boston, Boston, MA, 1999.

\bibitem{Watson95}
G.~N. Watson.
\newblock {\em A treatise on the theory of {B}essel functions}.
\newblock Cambridge Mathematical Library. Cambridge University Press,
  Cambridge, 1995.
\newblock Reprint of the second (1944) edition.

\end{thebibliography}

\end{document}